\documentclass[11pt]{article}

\usepackage{enumerate, amsmath, amsthm, amsfonts, amssymb, xy}
\usepackage[margin=1in]{geometry}
\usepackage{algorithm, algorithmic}

\input xy
\xyoption{all}
\usepackage{subfigure}

\newtheorem{theorem}{Theorem}[section]
\newtheorem{proposition}[theorem]{Proposition}
\newtheorem{lemma}[theorem]{Lemma}
\newtheorem{definition}[theorem]{Definition}
\newtheorem{example}[theorem]{Example}
\newtheorem{rmk}[theorem]{Remark}
\newtheorem{corollary}[theorem]{Corollary}
\newtheorem{conjecture}[theorem]{Conjecture}

\newenvironment{remark}{\begin{rmk}\rm}{\end{rmk}}
\numberwithin{equation}{section}

\newcommand{\N}{\mathbf{N}}

\newcommand{\Z}{\mathbf{Z}}

\renewcommand{\phi}{\varphi}
\renewcommand{\emptyset}{\varnothing}

\makeatletter
\def\Ddots{\mathinner{\mkern1mu\raise\p@
\vbox{\kern7\p@\hbox{.}}\mkern2mu
\raise4\p@\hbox{.}\mkern2mu\raise7\p@\hbox{.}\mkern1mu}}
\makeatother

\newcommand{\s}{\mathfrak{S}}
\newcommand{\co}{\operatorname{co}}
\title{A canonical expansion of the product of two Stanley symmetric functions}
\author{Nan Li}
\begin{document}
\maketitle

\begin{abstract}
We study the problem of expanding the product of two Stanley symmetric functions $F_w\cdot F_u$ into Stanley symmetric functions in some natural way. Our approach is to consider a Stanley symmetric function as a stabilized Schubert polynomial $F_{w}=\lim_{n\to \infty}\mathfrak{S}_{1^{n}\times w}$, and study the behavior of the expansion of $\s_{1^n\times w}\cdot\s_{1^n\times u}$ into Schubert polynomials, as $n$ increases. We prove that this expansion stabilizes and thus we get a natural expansion for the product of two Stanley symmetric functions. In the case when one permutation is Grassmannian, we have a better understanding of  this stability. We then study some other related stable properties, which provides a second proof of the main result.
\end{abstract}

\section{Introduction}
In \cite{RS}, Stanley defined a homogeneous power series $F_w$ in infinitely many variables $x_1,x_2,\dots$, to compute the number of reduced decompositions of a given permutation $w$. He also proved that $F_w$ is symmetric, and $F_w$ is now referred to as a Stanley symmetric function. Our convention is that $F_w$ means the usual $F_{w^{-1}}$ as defined in \cite{RS}. It is shown in \cite{bal} that  $$F_w=s_{D(w)},$$  where $D(w)$ is the diagram of $w$ and $s_{D(w)}$ is the generalized Schur function defined in terms of the column-strict balanced labellings of $D(w)$. We are interested in the problem of expanding the product of two Stanley symmetric functions $F_w\cdot F_u$ into Stanley symmetric functions. The hope is that we can explain the coefficients in terms of $D(w)$ and $D(u)$, as a generalized Littlewood-Richardson rule for Schur functions.

However, since the Stanley symmetric functions are not linearly independent, we want to expand them in some natural way. For $w\in S_m$ and $u\in S_n$, denote by $w\times u$ the permutation $v\in S_{m+n}$, with one line notation: $w(1)\cdots w(m) (u(1)+m)\cdots (u(n)+m)$. Also, by $1^n$, we mean $1\times 1\times \cdots \times 1=123\cdots n$. For example, $1^2\times 2134=124356$. We consider a Stanley symmetric function as a stabilized Schubert polynomial \cite{mac}:
\begin{equation}\label{limit}
F_{w}=\lim_{n\to \infty}\mathfrak{S}_{1^{n}\times w}.
\end{equation}

Divided difference operators were first used by Bernstein-Gelfand-Gelfand \cite{bern} and Demazure \cite{dema} for the study of the cohomology of flag manifolds. Later, Lascoux and Schu\"{u}tzenberger \cite{poly} developed the theory of Schubert polynomials based on divided difference operators. The collection $\{\s_w\mid w\in S_n\}$ of Schubert polynomials determines an integral basis for the cohomology ring of the flag manifold, and thus there exist integer structure constants $c_{wu}^{v}$ such that
\begin{equation*}
\s_w\cdot\s_u=\sum_{v}c_{wu}^v \s_v.
\end{equation*}
It is a long standing question to find a combinatorial description of these constants. Some special cases are known. The simplest but important case is Monk's rule \cite{monk}, which corresponds to the case when one of the Schubert polynomials is indexed by a simple transposition. A generalized Pieri rule was conjectured by Lascoux and Schu\"{u}tzenberger \cite{poly}, where they also sketched an algebraic proof. It was conjectured by Bergeron and Billey \cite{bb} in another form, and was proved by Sottile \cite{pieri} using geometry, and by Winkel \cite{combpieri} via a combinatorial proof. There are also results about the case of a Schubert polynomial times a Schur polynomial, for example see \cite{koh}, \cite{growth} and \cite{pre}.

In order to study the expansion of $F_w\cdot F_u$, we study the behavior, as $n$ increases, of the expansion of $\mathfrak{S}_{1^{n}\times w}\cdot \mathfrak{S}_{1^{n}\times u}$ into Schubert polynomials. Let us look at a toy example when $u=t_{m,m+1}$, a simple transposition.

By Monk's rule \cite{monk}, we have
$$\s_w\cdot \s_{t_{m,m+1}}=\sum_{\substack{j\le m<k\\ \ell(wt_{jk})=\ell(w)+1}}\s_{wt_{jk}},$$
where $\ell(w)$ is the length of the permutation $w$ and $wt_{jk}$ is the permutation obtained from $w$ by exchanging $w(j)$ and $w(k)$. Notice that $1\times t_{m,m+1}=t_{m+1,m+2}$.
Then for $\s_{1\times w}\cdot \s_{1\times t_{m,m+1}}$, we will have a term $\s_{1\times wt_{jk}}$ corresponding to each term $\s_{wt_{jk}}$ in the expansion of $\s_w\cdot \s_{t_{m,m+1}}$.  Let the position of $1$ in $w$ be $s$, i.e., $w^{-1}(1)=s$.  If $s\le m$, then there are no more permutations; otherwise, if $s>m$,  we get one more permutation $(1\times w)t_{1,s+1}$. This holds for all $\s_{1^n\times w}\cdot \s_{1^n\times t_{m,m+1}}$. More precisely, we have
$$\s_{1^n\times w}\cdot \s_{1^n\times t_{m,m+1}}=\sum_{\substack{j\le m<k\\ \ell(wt_{jk})=\ell(w)+1}}\s_{1^n\times wt_{jk}}(+\s_{1^{n-1}\times (1\times w)t_{1,s+1}}, \text{ if } s>m).$$
Now taking the limit for $n\to \infty$, we get the following canonical expansion:
$$F_{ w}\cdot F_{t_{m,m+1}}=\sum_{\substack{j\le m<k\\ \ell(wt_{jk})=\ell(w)+1}}F_{wt_{jk}}(+F_{(1\times w)t_{1,s+1}}, \text{ if } s>m).$$

Let us look at another example for $w=3241$ and $u=4312$. Consider $\s_{1^n\times 3241}\cdot\s_{1^n\times 4312}$ as $n$ increases. For $n=0,1,2$, we have
\begin{align*}
\s_{3241}\cdot\s_{4312}=&\s_{642135}\\
\s_{1\times 3241}\cdot\s_{1\times 4312}=&\s_{1\times 642135}+ \underline{\s_{265314}+\s_{2743156}+\s_{356214}+\s_{364215}+\s_{365124}}
    \\&\underline{+ \s_{462315}+\s_{561324}}\\
\s_{1^2\times 3241}\cdot\s_{1^2\times 4312}=&\s_{1^2\times 642135}
    + \s_{1\times 265314}+\s_{1\times 2743156}+\s_{1\times 356214}+\s_{1\times 364215}+\s_{1\times 365124}\\&+\s_{1\times 462315}+\s_{1\times 561324}+\underline{\s_{2375416}+\s_{246531}+\s_{256341}}.
    \end{align*}
Notice that as $n$ increases, we keep all the permutations appearing in the previous case and add some new permutations (the underlined terms). In this example, the expansion stabilizes after $n=2$, i.e., we do not add new permutations for $n>2$, i.e.,
   \begin{align*}
\s_{1^n\times 3241}\cdot\s_{1^n\times 4312}=&\s_{1^n\times 642135}
    + \s_{1^{n-1}\times 265314}+\s_{1^{n-1}\times 2743156}+\s_{1^{n-1}\times 356214}+\s_{1^{n-1}\times 364215}\\&+\s_{1^{n-1}\times 365124}+\s_{1^{n-1}\times 462315}+\s_{1^{n-1}\times 561324}+\s_{1^{n-2}\times 2375416}+\s_{1^{n-2}\times 246531}\\&+\s_{1^{n-2}\times 256341}.
   \end{align*}
   Then taking $n\to \infty$, we have
   \begin{align*}
F_{3241}\cdot F_{4312}=& F_{ 642135}+ F_{ 265314}+F_{ 2743156}+F_{356214}+F_{ 364215}+F_{365124}+F_{462315}+F_{ 561324}\\&+F_{2375416}+F_{246531}+F_{256341}.
\end{align*}
The stability of the expansion $\mathfrak{S}_{1^{n}\times w}\cdot \mathfrak{S}_{1^{n}\times u}$ we observed in the previous two examples are true in general. Here is the main result of this paper.
\begin{theorem}\label{main} Let $w,u$ be two permutations.
\begin{enumerate}
\item  Suppose $\s_w\cdot \s_u=\sum_{v_0\in V_0}c_{w,u}^{v_0}\s_{v_0}$.
Then $$\s_{1\times w}\cdot \s_{1\times u}=\sum_{v_0\in V_0}c_{w,u}^{v_0}\s_{1\times
v_0}+\sum_{v_1\in V_1}c_{w,u}^{v_1}\s_{v_1},$$ where $v_1(1)\neq 1$, for each $v_1\in V_1$.
\item Let $k=\ell(w)+\ell(u)$. Then for all $n\ge k$, we have
$$\mathfrak{S}_{1^{n}\times w}\cdot \mathfrak{S}_{1^{n}\times u}=\sum_{v_0\in V_0}c^{v_0}_{w,u}\s_{1^{n}\times v_0}+\sum_{v_1\in V_1}c^{v_1}_{w,u}\s_{1^{n-1}\times v_1}+\cdots+\sum_{v_k\in V_k}c^{v_k}_{w,u}\s_{1^{n-k}\times v_k},$$
where $V_i$ (possibly empty) is the set of new permutations appearing in $\s_{1^i\times w}\cdot\s_{1^i\times u}$ compared to $\s_{1^{i-1}\times w}\cdot\s_{1^{i-1}\times u}$.
Taking $n\to\infty$, we have a canonical expansion:
\begin{equation}\label{expansion}
F_w\cdot F_u=\sum_{v\in V}c^{v}_{w,u}F_v,
\end{equation}
 where $V=V_0\cup\cdots\cup V_k$.
\end{enumerate}
\end{theorem}
For a permutation $w\in S_n$, define the \textbf{code} $c(w)$ to be the sequence $c(w)=(c_1,c_2,\dots)$ of nonnegative integers given by $c_i=\#\{j\in [n]\mid j>i,\,w(j)<w(i)\}$. Define the length of $c(w)$ to be $i_0=\max\{i\mid c_i\neq 0\}$, denoted by $\ell(c(w))$. We call a permutation  \textbf{Grassmannian} if it has at most one descent. It is known that if $w$ is Grassmannian, then $\s_w$ is a Schur polynomial in $\ell(c(w))$ variables.
\begin{theorem}\label{full} Apply the above notations. If one of $w,u$ is Grassmannian, then we also have:
 \begin{enumerate}
 \item If $V_i=\emptyset$ for some $i$, then $V_j=\emptyset$ for all $j>i$. We call the smallest $i$ such that $V_i=\emptyset$ the \textbf{stability number} for $w,u$.
 \item The stability number is bounded by $\max\{\ell(c(w)),\ell(c(u))\}$. In particular, if $w=u$ with $w(1)\neq 1$, the stability number equals $w^{-1}(1)-1$.
 \end{enumerate}
\end{theorem}
\begin{conjecture}\label{conj}Theorem \ref{full} is true for general $w,u$.
\end{conjecture}
In Section 2, we prove Theorem \ref{main} using the combinatorial definition of Schubert polynomials given in \cite{comb}. In Section 3, we study the case when one of the permutation is Grassmannian. We prove Theorem \ref{main} and \ref{full} by an algorithm described in \cite{koh} using maximal transitions (\ref{max}).

For the case when both $w,u$ are Grassmannian, $\s_{1^n\times w}\cdot \s_{1^n\times u}$ is the product of two Schur polynomials for all $n$, so (\ref{expansion}) is described by the usual Littlewood-Richardson rule. When both $w$ and $u$ are 321-avoiding, by  \cite {comb}, (\ref{expansion}) gives an expansion of the product of two skew Schur functions. Compare this with the skew Littlewood-Richardson rule studied in \cite{assaf} and \cite{lam}, where they give a nice formula for the coefficients (with signs) in the expansion of two skew Schur functions into skew Schur functions. Here, we get all positive coefficients, but not all permutations appearing in the expansion are $321$-avoiding.


In Section 4, we generalize this stability to the product of double Schubert polynomials. We also give the definition of the weak and strong stable expansions,  and prove some other stable properties, which provide a second proof of Theorem \ref{main}.
\section{Proof of  Theorem \ref{main}}
Let us recall the combinatorial definition of Schubert polynomials introduced in Theorem 1,1~\cite{comb}. Let $p=\ell(w)$ be the length of $w$, and $R(w)$ be the set of all the reduced words of $w$. For $a=(a_1,\dots,a_p)$, let $K(a)$ be the set of all $a$-compatible sequences, i.e., $(i_1,\dots,i_p)$ such that: 1) $i_1\le\cdots\le i_p$; 2) $i_j\le a_j$, for $j=1,\dots,p$; and 3) $i_j<i_{j+1}$, if $a_j<a_{j+1}$. Then we have
\begin{equation}\label{combdef}
\s_w=\sum_{a\in R(w)}\sum_{(i_1,\dots,i_p)\in K(a)}x_{i_1}\cdots x_{i_p}.
\end{equation}

\begin{definition}\label{goodpair} For two integer vectors $b^1=(b^1_1,\dots,b^1_p)$ and $b^2=(b^2_1,\dots,b^2_p)$, consider the following conditions:
\begin{enumerate}
\item $b^1$ and $b^2$ are weakly increasing. Namely, $b^1_1\le \dots\le b^1_p$ and $b^2_1\le \dots\le b^2_p$.
\item $b^1$ is smaller than $b^2$, denoted by $b^1<b^2$, which means $b^1_i\le b^2_i$ for each $i=1,\dots,p$;
\item $b^1$ is similar with $b^2$, denoted by $b^1\sim b^2$, which means $b^1$ and $b^2$ increase at the same time, i.e., $b^1_i<b^1_{i+1}$ if and only if $b^2_i<b^2_{i+1}$;
\item $b^1$ and $b^2$ are bounded by $n$, i.e., $b^1_i\le n$ and  $b^2_i\le n$, for all $i=1,\dots,p$.
\end{enumerate}
We call $(b^1,b^2)$ a \textbf{good pair} if it satisfies the first three conditions, call it a \textbf{good-$n$ pair}, if all four conditions are satisfied.
\end{definition}
For example, $(b^1,b^2)$, with $b^1=(2,4,4,5)$ and $b^2=(2,6,6,8)$, is a good-$8$ pair. Denote $X_b=x_{b_1}x_{b_2}\cdots x_{b_p}$. For example, $X_{b^1}=x_2x_4^2x_5$, for the previous $b^1$. We use $\co(X_b)$ to denote the coefficient of $X_b$.
\begin{lemma}\label{good}
\begin{enumerate}
\item In $\s_w$, $\co(X_{b^1})\ge\co(X_{b^2})$, for any good pair $(b^1,b^2)$.
\item In $\s_{1^n\times u}$,  $\co(X_{b^1})=\co(X_{b^2})$, for any good-$n$ pair $(b^1,b^2)$.
\item In $\s_{1^n\times u}$,  $\co(X_{b^1}\cdot g)=\co(X_{b^2}\cdot g)$, for any good-$n$ pair $(b^1,b^2)$ and any monomial $g$ with variable indices larger than $n$.
\item In $\s_{1^n\times w}\cdot \s_{1^n\times u}$, $\co(X_{b^1}\cdot g)=\co(X_{b^2}\cdot g)$, for any good-$n$ pair $(b^1,b^2)$ any monomial $g$ with indices larger than $n$.
\end{enumerate}
\end{lemma}
\begin{proof}Parts 1-3 follow from the combinatorial definition (\ref{combdef}) of Schubert polynomials and Definition \ref{goodpair}. Now we will prove part 4. In fact, any $X_{b^1}\cdot g$ it is the product of two monomials, one from $\s_{1^n\times w}$ and one from $ \s_{1^n\times u}$, let us assume $X_{b^1}=X_{b^{11}}\cdot X_{b^{12}}$, and the corresponding decomposition for $X_{b^2}$ is $X_{b^2}=X_{b^{21}}\cdot X_{b^{22}}$. For example, consider the previous good-8 pair $(b^1,b^2)$. If $X_{b^1}=x_2x_4^2x_5=(x_2x_4)(x_4x_5)$ with $b^{11}=(2,4)$ and $b^{12}=(4,5)$, then we decompose $X_{b^2}=x_2x_6^2x_8$ as $(x_2x_6)(x_6x_8)$ with $b^{21}=(2,6)$ and $b^{22}=(6,8)$. Since $b^1\sim b^2$, we have $b^{11}\sim b^{21}$ and $b^{12}\sim b^{22}$. Applying part 3 to both pairs, we have $\co(X_{b^1}\cdot g)=\co(X_{b^2}\cdot g)$.
\end{proof}

Write the code of $w$ as $c(w)=(c_1,c_2,\dots,c_p)$ and $X^{c(w)}=x_1^{c_1}x_2^{c_2}\cdots x_p^{c_p}$. Let $b(c)$ be the weakly increasing sequence such that $X_{b(c)}=X^c$. We use reverse lex-order in this section. It is known that the top degree term of $\s_w$ is $X^{c(w)}$, i.e.,
\begin{equation}\label{top}\s_w=X^{c(w)}+\sum_{b}X_b,
\end{equation}
where each $b$ satisfies $b<b(c(w))$ termwisely, as defined in part 2 of Definition \ref{goodpair}. Now we consider the process of getting the expansion of $\s_w\cdot\s_u$. By (\ref{combdef}), the top degree term is $X^{c(w)+c(u)}$. Let $v_1$ be the permutation such that $c(v_1)=c(w)+c(u)$. Then $$\s_w\cdot\s_u=\s_{v_1}+\cdots,$$
so $c_{wu}^{v_1}=1$. Then consider the top degree term in $\s_w\cdot\s_u-\s_{v_1}$. Let it be $c_2X^{c(v_2)}$ for some $v_2$. Then $$\s_w\cdot\s_u-\s_{v_1}=c_2\s_{v_2}+\cdots.$$
Next, consider the top degree term in $\s_w\cdot\s_u-\s_{v_1}-c_2\s_{v_2}$, etc. Since there are finitely many monomials in $\s_w\cdot\s_u$, this process terminates, and we get an expansion $\s_w\cdot\s_u=\sum_{v\in V_0}c_{wu}^v\s_v$.

\begin{proof}[of  Theorem \ref{main}]
\begin{enumerate}
\item By the combinatorial definition of Schubert polynomial (\ref{combdef}) and the above process of expanding $\s_w\cdot\s_u$, we have $c_{1\times w,1\times u}^{1\times v}=c_{w,u}^v$ for all $v\in V_0$. Further more, each term in $$\s_{1\times w}\cdot\s_{1\times u}-\sum_{v_0\in V_0}c_{w,u}^{v_0}\s_{1\times v_0}$$ is divided by $x_1$. So any $\s_v$ with $c(v)=(c_1,c_2,\dots)$ appear in the above difference has $c_1\neq 0$, which is equivalent to $v(1)\neq 0$. This proves part one.
\item
For a fixed $n$, suppose  $$\s_{1^n\times w}\cdot \s_{1^n\times u}=\sum_{v\in V}c^v_{wu}\s_v.$$
We claim that the code $c(v)=(c_1,c_2,\dots,c_p)$ for $v\in V$ has to satisfy the following property: let $c(v)_n=(c_1,c_2,\dots,c_n)$ be the first  $n$ elements in $c(v)$. Let $i(v)$ be the smallest number such that $c_i\neq 0$. Then the claim is that if $i(v)\le n$, then for all $i(v)<j\le n$, we have $c_j\neq 0$. Suppose we have proved this claim. Then since $c_1+\cdots+c_n\le k$, where $k=\ell(w)+\ell(u)$, for each $v\in V$, we have $i(v)\ge n-k$. In other words, the code $c(v)$  starts with at least $n-k$ zeros, and thus $v$ starts with $12\cdots (n-k)$, which will finish the proof. Now let us prove the claim.

In fact, suppose we have some $v_0\in V$ which does not satisfy the claim. Namely there exists some $j$ such that $i(v)<j\le n$ and $c_j=0$. Let $c'=(0,c_1,c_2,\dots,c_{j-1},c_{j+1},\dots,c_n)$. Consider the pair $b^1=b(c(v)_n)$ and $b^2=b(c')$, i.e., $X_{b^1}=X^{c(v)_n}$ and $X_{b^2}=X^{c'}$. For example, let $n=7$, and  $c(v_0)_n=(0,0,0,2,3,0,2)$. Then $X_{b^1}=X_4^2X_5^3x_7^2$, $c'=(0,0,0,0,2,3,2)$ and $X_{b^2}=X_5^2X_6^3x_7^2$. Then $(b^1, b^2)$ is a good $n$-pair.

Now let $g=X^{(c_{n+1},\dots,c_p)}$. Notice that $X_{b^1}\cdot g$ is the top degree term in $\s_{v_0}$ by (\ref{top}). Since $b^2>b^1$, $\co(X_{b^2}\cdot g)=0$ in $\s_{v_0}$. Therefore, $\co(X_{b^1}\cdot g)>\co(X_{b^2}\cdot g)$ in $\s_{v_0}$. By Lemma \ref{good}, on the right hand side, for each $v\in V$, we have $\co(X_{b^1}\cdot g)\ge\co(X_{b^2}\cdot g)$, therefore, on the right hand side, we have $\co(X_{b^1}\cdot g)>\co(X_{b^2}\cdot g)$. However,  on the left hand side, we must have $\co(X_{b^1}\cdot g)=\co(X_{b^2}\cdot g)$, a contradiction.\end{enumerate}\end{proof}

\section{Schubert polynomial times a Schur polynomial}

In this section we will prove Theorem \ref{main} and \ref{full} for the case when one of the permutation $w,u$ is Grassmannian. We will apply an algorithm for multiplying a Schubert polynomial by a Schur polynomial based on the following result. This result was originally proved using Kohnert's algorithm, which unfortunately, has not been completely proved yet. However, using the very similar algorithm called  ladder and chute moves studied in \cite{bb}, we can still show that the following theorem is true.
\begin{theorem}[Theorem 3.1 in {\cite{koh}}]
 Let $\s_u$ be a Schur polynomial with $m$ variables, i.e., $u$ is a Grassmannian permutation with $\ell(c(u))=m$. Let $\s_w$ be a Schubert polynomial with $m$ variables, i.e., $\ell(c(w))=m$. Then $$\s_w\cdot \s_u=\s_{w\times u}\downarrow A_m,$$
where $f\downarrow A_m=f(x_1,\dots,x_{m},0,\dots,0)$.
\end{theorem}
  The algorithm we will apply for multiplying a Schubert polynomial by a Schur polynomial was studied in \cite{koh} and is a modification of the algorithm by Lascoux and Sch\"{u}tzenberger \cite{LS} for decomposing the product of two Schur functions into a sum of Schur functions.
\subsection{Maximal transition tree}
Recall that $wt_{rs}$ is the permutation obtained from $w$ by switching $w(r)$ and $w(s)$. Let $r$ be the largest descent of the permutation $w$, and $s$ be the largest integer such that $w(s)<w(r)$. The following formula follows from Monk's rule \cite{monk}
 \begin{equation}\label{max}\s_{w}=x_r\s_{u}+\sum_{v\in S(w)}\s_{v},
  \end{equation}where $u=wt_{rs}$ and $S(w)$ is the set of permutations of the form $wt_{rs}t_{jr}$ with $j<r$ such that $\ell(wt_{rs}t_{jr})=\ell(w)$. So each $v\in S(w)$ corresponds to a different $j\in J(w)$. We call (\ref{max}) a \textbf{maximal transition} (MT for short) (see \cite{LS}). For example, for $w=321654$, we have $r(w)=5$, $s(w)=6$, $J(w)=\{1,2,3\}$ and $S(w)=\{421635,341625,324615\}$. We call each $v\in S(w)$ a descendent of $w$.

 Notice that $c_i=0$, for all $i>r(w)$ in the code $c(w)=(c_1,c_2,\dots)$, and $\s_w$ is a polynomial with $r(w)$ variables. So if $r(w)\le m$, then $\s_w=\s_w\downarrow A_m$. If $r(w)>m$, we have $\s_{w}\downarrow A_m=\sum_{v\in S(w)}\s_{v}\downarrow A_m$ by (\ref{max}), since we set $x_r=0$. Notice that for each permutation $v\in S(w)$, $r(v)<r(w)$. We call a permutation $v$ \textbf{bad} if $v^{-1}(1)>m+1$. If $v$ is bad, then $x_{m+1}$ divides each monomial of $\s_v$, so $\s_{v}\downarrow A_m=0$.

   Apply MT successively to $w\times u$, each $v\in S(w\times u)$ and their descendants as long as the permutation is not bad, until their largest descents are smaller than $m$. This way we get a finite tree with two types of leaves:  1) a permutation with largest descent $\le m$, we call it a \textbf{good leaf}; and 2) a bad permutation as defined above. Then $\s_{w\times u}\downarrow A_m$ is obtained by summing up all of the good leaves. We call this tree the MT-tree rooted at $w\times u$; we call the edge between a permutation $w$ and one of its descendant $v\in S(w)$ an MT-move.
  \begin{example}\label{tree}Here is an example of the MT tree rooted at $w\times u$, for  $w=321$, $u=2413$ and $m=2$ (see Figure \ref{figtree}). The leaves we cross out are the bad leaves, i.e., permutations with 1 in position larger than $m+1=3$. The remaining leaves are good leaves, i.e., they have largest descent $\le m=2$.
  \begin{figure}
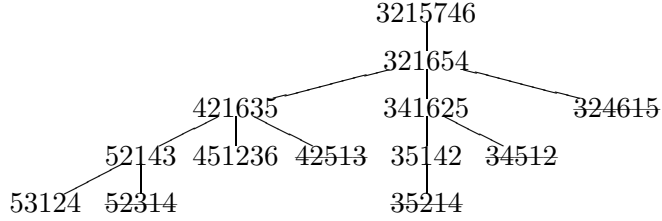

$$\xy 0;/r.30pc/:
(0,0)*{3215746}="a";
(0,-5)*{321654}="b";
(-20,-10)*{421635}="c";
(0,-10)*{341625}="d";
(20,-10)*{\rlap{---------}{324615}}="e";
(-30,-15)*{52143}="f";
(-20,-15)*{451236}="g";
(-10,-15)*{\rlap{--------}{42513}}="h";
(0,-15)*{35142}="i";
(10,-15)*{\rlap{--------}{34512}}="j";
(-40,-20)*{53124}="k";
(-30,-20)*{\rlap{--------}{52314}}="l";
(0,-20)*{\rlap{--------}{35214}}="m";
 "a"; "b"**\dir{-};
  "b"; "c"**\dir{-};
   "b"; "d"**\dir{-};
    "b"; "e"**\dir{-};
     "c"; "f"**\dir{-};
      "c"; "g"**\dir{-};
       "c"; "h"**\dir{-};
        "f"; "k"**\dir{-};
         "f"; "l"**\dir{-};
          "d"; "i"**\dir{-};
           "d"; "j"**\dir{-};
            "i"; "m"**\dir{-};
           \endxy$$
           \caption{MT-tree rooted at $321\times 2413$ for Example \ref{tree}}
           \label{figtree}
           \end{figure}
So summing up all the good leaves, we have $\s_{321}\cdot\s_{2413}=\s_{321\times 2413}\downarrow A_2=\s_{53124}+\s_{45123}$.
  \end{example}
\begin{remark}Notice that in Figure \ref{figtree}, the descendants of $341625$ are bad leaves ($35214$ and $34512$). It will be nice if one could simplify the tree so that we can remove $341625$ without applying further moves. However, it seems that such a rule, if exists, will be related with some pattern avoidances, which is hard to describe in general.
\end{remark}

Now we want to study the difference between the MT-tree rooted at $1\times w\times 1\times u$ and the one rooted at $w\times u$.
\begin{example}\label{tree2}Continue Example \ref{tree}. We study $\s_{1\times 321}\cdot\s_{1\times2413}$ (see Figure \ref{figtree2}). Notice that now $m=3$ instead of $2$ in Example \ref{tree}.
\begin{figure}
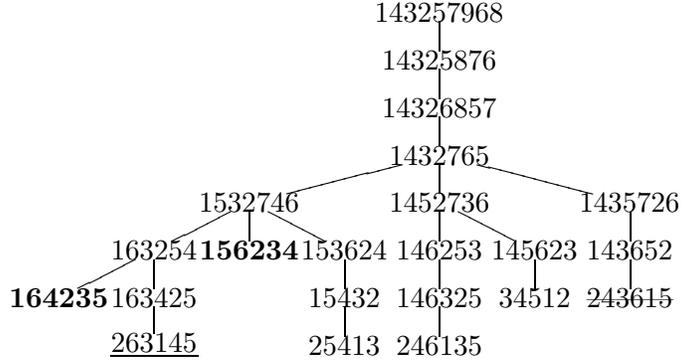

$$
\xy 0;/r.30pc/:
(0,10)*{143257968}="hh";
(0,5)*{14325876}="ii";
(0,0)*{14326857}="a";
(0,-5)*{1432765}="b";
(-20,-10)*{1532746}="c";
(0,-10)*{1452736}="d";
(20,-10)*{1435726}="e";
(-30,-15)*{163254}="f";
(-20,-15)*{\textbf{156234}}="g";
(-10,-15)*{153624}="h";
(0,-15)*{146253}="i";
(10,-15)*{145623}="j";
(-40,-20)*{\textbf{164235}}="k";
(-30,-20)*{163425}="l";
(0,-20)*{146325}="m";
(-10,-20)*{15432}="aa";
(10,-20)*{34512}="bb";
(20,-15)*{143652}="cc";
(20,-20)*{\rlap{---------}{243615}}="dd";
(-10,-25)*{25413}="ee";
(0,-25)*{246135}="ff";
(-30,-25)*{\underline{263145}}="gg";
 "a"; "b"**\dir{-};
  "b"; "c"**\dir{-};
   "b"; "d"**\dir{-};
    "b"; "e"**\dir{-};
     "c"; "f"**\dir{-};
      "c"; "g"**\dir{-};
       "c"; "h"**\dir{-};
        "f"; "k"**\dir{-};
         "f"; "l"**\dir{-};
          "d"; "i"**\dir{-};
           "d"; "j"**\dir{-};
            "i"; "m"**\dir{-};
"hh"; "ii"**\dir{-};
"ii"; "a"**\dir{-};
"aa"; "h"**\dir{-};
"aa"; "ee"**\dir{-};
"bb"; "j"**\dir{-};
"cc"; "e"**\dir{-};
"cc"; "dd"**\dir{-};
"l"; "gg"**\dir{-};
"m"; "ff"**\dir{-};
\endxy
$$
\caption{MT-tree rooted at $1\times 321\times 1\times 2413$ for Example \ref{tree2}}
\label{figtree2}
\end{figure}
 Summing up all good leaves, we have $\s_{1\times 321}\cdot\s_{1\times2413}=\s_{1432\times 13524}\downarrow A_3=\s_{164235}+\s_{156234}+\s_{263145}+\s_{25413}+\s_{246135}+\s_{34512}$.
\end{example}

Compare the leaves of the above tree and those in Example \ref{tree}. We have the following observations.
\begin{enumerate}
\item The good leaves in Example \ref{tree} ($53124$ and $45123$) stay good in Example \ref{tree2}, simply with a one added in front ( $1\times53124=164235$ and $1\times45123=156234$, bolded in Figure \ref{figtree2}).
\item The remaining good leaves in Example \ref{tree2} are descendants of some bad leaves in Example \ref{tree}. For example, $263145$ (underlined in Figure \ref{figtree2}) is obtained from $52314$ which used to be bad in Example \ref{tree}.
\item For the new good leaves in Example \ref{tree2}, the position of 1 stays the same as their ancestor in Example \ref{tree}. For example, both $263145$ and $52314$ has 1 in the fourth position.
\end{enumerate}
In general, the first and second observations above are true as a consequence of Lemma \ref{same} (we will prove it in the next subsection), and the third observation is true by Lemma \ref{one}.
\begin{lemma}\label{same}For the same $m$ as for $w,u$, the leaves (``good" and ``bad")
of $w\times 1\times u$ are the same as leaves of $w\times u$.
\end{lemma}
\begin{lemma}\label{one}For any \textbf{reduced} permutation $w$ (cannot make more MT-moves), if we add $1$ in the
beginning and then apply the MT-moves to $1\times w$, the position of 1 in the leaves is the same as the position of $1$ in $w$.
\end{lemma}
\begin{proof} Let $r_0$ be the last descent of $w$, $s_0$ be the largest number such that $w(s_0)<w(r_0)$.
$w$ is reduced implies that in $w$, we can not find any $j$ such that $j<r_0$ and $w(j)<w(s_0)$. Then since $r_0$ is the last
descent, we can see that $w(r_0+1)=1$. Then in
the first move for $1\times w$, we will have $j=1$, $r=r_0+1$ thus and move $1$ to the position of $r_0+1$.
 After this move, all the rest will not change the position of $1$. So $1$ will be in the position $r_0+1$, which is the same as the position
of $1$ in $w$. After this move, all the rest will not change the position of $1$.
\end{proof}

 Now notice that in Example \ref{tree2}, there is still one bad leaf $243615$ (see Figure \ref{figtree2}). So in the next step $\s_{1^2\times 321}\cdot\s_{1^2\times 2413}$, there will be some more good leaves with $243615$ as ancestor. After that, the expansion $\s_{1^n\times 321}\cdot\s_{1^n\times 2413}$, for $n\ge 2$ should have no more new permutations. And in fact, this is the case:
$\s_{1^n\times 321}\cdot\s_{1^n\times 2413}=\s_{1^n\times 53124}+\s_{1^n\times 45123}+\s_{1^{n-1}\times263145}+\s_{1^{n-1}\times25413}+\s_{1^{n-1}\times246135}+\s_{1^{n-1}\times34512}+\s_{1^{n-2}\times 236415}$, for all $n\ge 2$.
So we have
\begin{equation}\label{sta}
F_{321}\cdot F_{2413}=F_{53124}+F_{45123}+F_{263145}+F_{25413}+F_{246135}+F_{34512}+F_{236415}.
\end{equation}
So the stability number for $\s_{321}\cdot\s_{2413}$ is $2$, as predicted by Theorem \ref{full} part 2 that it should be bounded by $\ell(c(321))=\ell(c(2413))=2$. Now look at the positions of 1 in each permutation appearing on the right hand side of (\ref{sta}): $I=\{3,4,5\}$, which is an interval without any gaps. In general, we have
\begin{lemma}\label{gap} Let $F_w\cdot F_u=\sum_{v\in V}F_v$ be the expansion we get by Theorem \ref{main}. Let $I\{v^{-1}(1)\mid v\in V\}$. Then $I=[a,b]$ an interval without any gaps.
\end{lemma}
 Lemma \ref{gap} together with Theorem \ref{main} will imply Theorem \ref{full}. For a proof of Lemma \ref{gap}, we also want to use the diagrams interpretation of the MT-move studied in the next subsection.
\begin{lemma}\label{reduced}If a permutation is reduced, then there are no descents after $1$. In other words, the length of the code is the number of boxes in the first column.
\end{lemma}

\begin{proof} Suppose there is a descent after $1$, then it is not hard to see that this permutation is not reduced, since there must exists a $j$ for which we can apply MT-move.
\end{proof}

Now assume that $\s_u$ is a Schur polynomial and $\s_w$ is a Schubert polynomial both in $m$ variables. Use the MT algorithm, we can show the result in both Theorem \ref{main} and Theorem \ref{full}.
\begin{proof}[proof of Theorem \ref{main}]
\begin{enumerate}
\item Consider the expansion of $\s_{1\times w}\times \s_{1\times u}$ by looking at the tree rooted
at $1\times w\times 1\times u$. By definition, the good leaves of the tree rooted at $v\times u$ has last descent $\le m$ and position of the letter $1$ is not larger than $m+1$. By Lemma \ref{same}, leaves of $v\times  1\times u$ are the same as the leaves for  $v\times u$, both good leaves and
bad leaves. Then it is not hard to see that leaves of $1\times v\times 1\times u$ are
just leaves of $v\times  1\times u$ with an $1$ appended to the front. Now $n=1+\ell$, so good leaves of $v\times u$
are still good leaves for $1\times v\times 1\times u$, just with an $1$ appended to the front. Now for those bad leaves,
by Lemma \ref{one}, after we append $1$ to the front, and continue to apply the moves,
 the position of $1$ will not change. But because the number of variables in $\s_{1\times w}$ and $\s_{1\times u}$ are $m+1$ instead of $m$ now, some of the bad leaves (when position
 of $1$ is $m+2$) will become good. Moveover these newly added good leaves will not start with $1$. So the Schubert polynomials
 appeared in the expansion of $X_{1\times v}\times X_{1\times u}$ are indexed by
 the old good leaves, which all start with $1$, and possibly some new good leaves, which all do not start with $1$. This shows the first part.
\item When we append more ones in front of $u$ and $v$, all leaves of the tree will become good leaves, but there are only finitely many of them.
 So finally the expansion will be stable.
  \end{enumerate}
\end{proof}

\begin{proof}[Proof of Theorem \ref{full}]
\begin{enumerate}
 \item By MT algorithm, if we are able to move a portion with length $s_1$ and $s_2$ to the first column, then we are able to move a portion with any length between $s_1$ and $s_2$. So there is no gap when new terms showing up in the expansion when $n$ increases. Applying different choices of $j$ gives us different number of potential boxes to be added to the first column. But this number should have no gap: if there are cases when there are one boxes left and four boxes left. Then there must be some combination of choices of $j$'s such that there are two and three boxes left to be added to the first column.
\item By Lemma \ref{reduced}, for all leaves, the length of the code is the length of the first column. It is clear that the longest possible first column is the sum of the length of $c(w)$ and $c(u)$. And the stable number bounded by the maximal length of the first column minus $c(w)$, so we get a bound by the maximal length of $c(w)$ and $c(u)$. In the case $w=u$, and both being Grassmannian, we have $m=w^{-1}(1)-1$, which is exactly the stable number.

 \end{enumerate}
\end{proof}

 \subsection{MT-move in terms of diagrams}
 In order to prove Lemma \ref{same} and Lemma \ref{gap}, we want to describe the MT-move in terms of diagrams.

First, there is a correspondence
between the set of inversions of $w$ and the boxes in the diagram. An inversion in a permutation $w$ is a pair of $(i,j)$ such that $i<j$ and $w^{-1}(i)>w^{-1}(j)$. We denote a box of the diagram in the $i$th row and $j$th column by $B_{ij}$. Then the box $B_{ij}$ corresponds to the inversion $(j,w(i))$ in $w$. For example, here is the diagram for $w=3215746$ (see Figure \ref{w}). The box $B_{56}$ (indicated by a bullet) corresponds to the inversion $(6,7)$ in $w$.
\begin{figure}[htp]
  \begin{center}
    \subfigure[w]{\label{w}$\,\,\,\,$ $\xy 0;/r.6pc/:
(-0.5,6.5)*{3};
(-0.5,5.5)*{2};
(-0.5,4.5)*{1};
(-0.5,3.5)*{5};
(-0.5,2.5)*{7};
(-0.5,1.5)*{4};
(-0.5,0.5)*{6};
(0,0)="(0,0)";
(1,0)="(1,0)";
(2,0)="(2,0)";
(3,0)="(3,0)";
(4,0)="(4,0)";
(5,0)="(5,0)";
(6,0)="(6,0)";
(7,0)="(7,0)";
(1,7)="(1,7)";
(2,7)="(2,7)";
(3,7)="(3,7)";
(4,7)="(4,7)";
(5,7)="(5,7)";
(6,7)="(6,7)";
(7,7)="(7,7)";
(0,1)="(0,1)";
(0,2)="(0,2)";
(0,3)="(0,3)";
(0,4)="(0,4)";
(0,5)="(0,5)";
(0,6)="(0,6)";
(0,7)="(0,7)";
(7,1)="(7,1)";
(7,2)="(7,2)";
(7,3)="(7,3)";
(7,4)="(7,4)";
(7,5)="(7,5)";
(7,6)="(7,6)";
(7,7)="(7,7)";
 "(0,0)"; "(7,0)"**\dir{.};
  "(0,1)"; "(7,1)"**\dir{.};
"(0,2)"; "(7,2)"**\dir{.};
"(0,3)"; "(7,3)"**\dir{.};
"(0,4)"; "(7,4)"**\dir{.};
"(0,5)"; "(7,5)"**\dir{.};
"(0,6)"; "(7,6)"**\dir{.};
"(0,7)"; "(7,7)"**\dir{.};
"(0,0)"; "(0,7)"**\dir{.};
"(1,0)"; "(1,7)"**\dir{.};
"(2,0)"; "(2,7)"**\dir{.};
"(3,0)"; "(3,7)"**\dir{.};
"(4,0)"; "(4,7)"**\dir{.};
"(5,0)"; "(5,7)"**\dir{.};
"(6,0)"; "(6,7)"**\dir{.};
"(7,0)"; "(7,7)"**\dir{.};
(0.5,0)="(0.5,0)";
(1.5,0)="(1.5,0)";
(2.5,0)="(2.5,0)";
(3.5,0)="(3.5,0)";
(4.5,0)="(4.5,0)";
(5.5,0)="(5.5,0)";
(6.5,0)="(6.5,0)";
(7,0.5)="(7,0.5)";
(7,1.5)="(7,1.5)";
(7,2.5)="(7,2.5)";
(7,3.5)="(7,3.5)";
(7,4.5)="(7,4.5)";
(7,5.5)="(7,5.5)";
(7,6.5)="(7,6.5)";
(2.5,6.5)*{\circ}="w1";
(1.5,5.5)*{\circ}="w2";
(0.5,4.5)*{\circ}="w3";
(4.5,3.5)*{\circ}="w4";
(6.5,2.5)*{\circ}="w5";
(3.5,1.5)*{\circ}="w6";
(5.5,0.5)*{\circ}="w7";
"w1"; "(2.5,0)"**\dir{-};
"w1"; "(7,6.5)"**\dir{-};
"w2"; "(1.5,0)"**\dir{-};
"w2"; "(7,5.5)"**\dir{-};
"w3"; "(0.5,0)"**\dir{-};
"w3"; "(7,4.5)"**\dir{-};
"w4"; "(4.5,0)"**\dir{-};
"w4"; "(7,3.5)"**\dir{-};
"w5"; "(6.5,0)"**\dir{-};
"w5"; "(7,2.5)"**\dir{-};
"w6"; "(3.5,0)"**\dir{-};
"w6"; "(7,1.5)"**\dir{-};
"w7"; "(5.5,0)"**\dir{-};
"w7"; "(7,0.5)"**\dir{-};
(5.5,2.5)*{\bullet};
\endxy
$ $\,\,\,\,$}
    \subfigure[v]{\label{v}$\,\,\,\,$  $
\xy 0;/r.6pc/:
(-0.5,6.5)*{3};
(-0.5,5.5)*{2};
(-0.5,4.5)*{1};
(-0.5,3.5)*{6};
(-0.5,2.5)*{5};
(-0.5,1.5)*{4};
(-0.5,0.5)*{7};
(0,0)="(0,0)";
(1,0)="(1,0)";
(2,0)="(2,0)";
(3,0)="(3,0)";
(4,0)="(4,0)";
(5,0)="(5,0)";
(6,0)="(6,0)";
(7,0)="(7,0)";
(1,7)="(1,7)";
(2,7)="(2,7)";
(3,7)="(3,7)";
(4,7)="(4,7)";
(5,7)="(5,7)";
(6,7)="(6,7)";
(7,7)="(7,7)";
(0,1)="(0,1)";
(0,2)="(0,2)";
(0,3)="(0,3)";
(0,4)="(0,4)";
(0,5)="(0,5)";
(0,6)="(0,6)";
(0,7)="(0,7)";
(7,1)="(7,1)";
(7,2)="(7,2)";
(7,3)="(7,3)";
(7,4)="(7,4)";
(7,5)="(7,5)";
(7,6)="(7,6)";
(7,7)="(7,7)";
 "(0,0)"; "(7,0)"**\dir{.};
  "(0,1)"; "(7,1)"**\dir{.};
"(0,2)"; "(7,2)"**\dir{.};
"(0,3)"; "(7,3)"**\dir{.};
"(0,4)"; "(7,4)"**\dir{.};
"(0,5)"; "(7,5)"**\dir{.};
"(0,6)"; "(7,6)"**\dir{.};
"(0,7)"; "(7,7)"**\dir{.};
"(0,0)"; "(0,7)"**\dir{.};
"(1,0)"; "(1,7)"**\dir{.};
"(2,0)"; "(2,7)"**\dir{.};
"(3,0)"; "(3,7)"**\dir{.};
"(4,0)"; "(4,7)"**\dir{.};
"(5,0)"; "(5,7)"**\dir{.};
"(6,0)"; "(6,7)"**\dir{.};
"(7,0)"; "(7,7)"**\dir{.};
(0.5,0)="(0.5,0)";
(1.5,0)="(1.5,0)";
(2.5,0)="(2.5,0)";
(3.5,0)="(3.5,0)";
(4.5,0)="(4.5,0)";
(5.5,0)="(5.5,0)";
(6.5,0)="(6.5,0)";
(7,0.5)="(7,0.5)";
(7,1.5)="(7,1.5)";
(7,2.5)="(7,2.5)";
(7,3.5)="(7,3.5)";
(7,4.5)="(7,4.5)";
(7,5.5)="(7,5.5)";
(7,6.5)="(7,6.5)";
(2.5,6.5)*{\circ}="w1";
(1.5,5.5)*{\circ}="w2";
(0.5,4.5)*{\circ}="w3";
(5.5,3.5)*{\circ}="w4";
(4.5,2.5)*{\circ}="w5";
(3.5,1.5)*{\circ}="w6";
(6.5,0.5)*{\circ}="w7";
"w1"; "(2.5,0)"**\dir{-};
"w1"; "(7,6.5)"**\dir{-};
"w2"; "(1.5,0)"**\dir{-};
"w2"; "(7,5.5)"**\dir{-};
"w3"; "(0.5,0)"**\dir{-};
"w3"; "(7,4.5)"**\dir{-};
"w4"; "(5.5,0)"**\dir{-};
"w4"; "(7,3.5)"**\dir{-};
"w5"; "(4.5,0)"**\dir{-};
"w5"; "(7,2.5)"**\dir{-};
"w6"; "(3.5,0)"**\dir{-};
"w6"; "(7,1.5)"**\dir{-};
"w7"; "(6.5,0)"**\dir{-};
"w7"; "(7,0.5)"**\dir{-};
(4.5,3.5)*{\bullet};
\endxy
$  $\,\,\,\,$}
  \end{center}
  \caption{MT-move}
  \label{fig:edge}
\end{figure}

Now we study the MT-move in terms of diagrams. Let $v$ be a descendant of $w$ via an MT-move. Then $D(v)$ is obtained from $D(w)$ by moving some part of the diagram up and left. For example, as shown in the first step of Example  \ref{tree}, applying an MT-move to $w=3215746$,  we get  $v=3216547$, and the diagram of  $v$  is obtained from $D(w)$ by moving the box with a bullet up and left by one row and one column (see Figure \ref{fig:edge}). Notice that this diagram move is very similar to the move described in \cite{kmove}.

Recall that $v=wt_{rs}t_{sj}$, where $r$ is the largest descent of $w$, $s$ is the largest number $s>r$ such that $w(s)<w(r)$, and $j$ is some number $j<r$ such that $wt_{rs}t_{sj}$ has the same length as $w$. From $w$ to $v=wt_{rs}t_{sj}$, we have the following change of inversions:
\begin{enumerate}
\item change each inversion $(w(i),w(s))$ to an inversion $(w(i),w(j))$, for $j<i\le r$. In terms of diagrams, this corresponds to moving the part of column $w(s)$ left to column $w(j)$, where the part is from the $(j+1)$th row to the $r$th row.
\item change each inversion $(w(r), w(i))$ to $(w(j), w(i))$, for $r<i<s$ and $w(i)>w(j)$. In terms of diagrams, this corresponds to moving the part of row $r$ with column indices in $\{w(i)\mid w(i)>w(j),\,r<i<s\}$ up to row $j$.
\end{enumerate}
In the diagram of $w$, consider the right-down corner box $B(w)$, i.e., the box in the rightmost column of the lowest row. By the definition of $r(w)=r$ and $s(w)=s$, we have $B(w)=B_{r,w(s)}$. For each $j\in J(w)$, denote the box $B_{j,w(j)}$ by $T(w,j)$. Then the above changes of inversions can be seen as moving some blocks with $B(w)$ as its right-down corner up and left so that $T(w,j)$ becomes its up-left corner. For example, consider $w=321654$ in the branching part of Example $\ref{tree}$, with $J(w)=\{1,2,3\}$. See Figure \ref{anc} for $D(w)$, where $B(w)$ is marked with a bullet and all three possible $T(w,j)$'s are marked with $\times$. Now applying MT-moves to $D(w)$, all three $D(v)$, for $v\in S(w)$ are shown in Figure \ref{j1}, \ref{j2} and \ref{j3}. Using this diagram interpretation of the MT-move, we can prove Lemma \ref{same} by comparing the MT-moves of $D(w\times u)$ and $D(1\times w\times 1\times u)$.

\begin{proof}[proof of Lemma \ref{same}]
Compare the diagram of $w\times u$ and $w\times 1\times u$ (see Example \ref{ab}). The two diagrams are basically the same, but because of the ``$1$" in the middle of $w\times 1\times u$, the boxes corresponding to $u$ are down by one row and right by one column. We call them delayed boxes. Now we apply the maximal transitions to $t:=w\times 1\times u$ and compare the moves to those for $w\times u$. As in the previous discussion, we start from the rightmost box $B$ (marked with a bullet) in the lowest row, which is in row $r$ and column $w(t)$. We will move some part of the diagram with $B$ as as its right-down corner of the diagram up and left so that the box $T$ in  row $j$ and column $w(t)$ becomes its up-left corner. Compare each move of $s$ to the moves of $w\times u$. There are two cases:
\begin{enumerate}
\item the corresponding box for $B$ in $w\times u$ is also the rightmost box in the lowest row. In Example \ref{ab}, $A_i$ and the last $B_{ij}$ in its row has the same $B$ box ($A_1$ and $B_{12}$, $A_2$ and $B_{22}$, etc.).
\item In the lowest row, some delayed boxes are to the right of the box $B$ for $w\times u$. In Example \ref{ab}, every $B_{ij}$ not belongs to the previous case is in this case.
\end{enumerate}

Therefore, as we apply maximal transition to $w\times 1\times u$, if the delayed boxes are not on the way, we can apply the same move as for  $w\times u$; if we are not so lucky, we need to clear our way by moving all the delayed boxes up and left first. There are two important things to notice for this case: 1) there is only one possible $j$ to use (and always $j>m$), and it is exactly one row up and one column left to the delayed block that we need to move. 2) after moving this delayed block up and left by one, the boxes in this block are no longer delayed. In other words, this cleaning work will not affect the actually moving work, and this cleaning work is finite. Once we finish all the cleaning work, we will get the exactly the same permutation in the process of $w\times u$, as in the example $A_5=B_{52}$.  

 Now it is left to show that all the cleaning work can be done before we get to the leaves of $w\times u$. Consider the condition when we get to a leave: 1) all boxes are above the $(m+1)$th row (good leaves) or 2) the first column has more than $m$ boxes (bad leaves). If there are still some delayed boxes, since in the process of cleaning $j>m$, it is not possible that all boxes are above the $(m+1)$th row. In the second case, assume that the first column already has more than $m$ boxes, but we still have delayed boxes. Since these delayed boxes are not in the first column, we can still do the cleaning until there are no delayed boxes left.
\end{proof}

\begin{example}\label{ab}Here is an example for $w=1^2\times32154$ and $u=698435127$ with $m=6$. We start from $A_1=w\times u$ and $B_{11}=w\times 1\times u$. Here are their diagrams (for simplicity, we ignore the $1^2$ in the front. )
$$A_1:\,\,
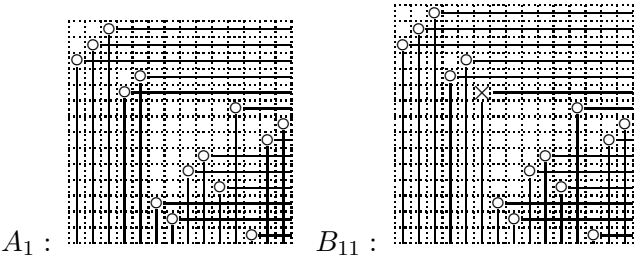
\xy 0;/r.5pc/:
(0,0)="(0,0)";
(1,0)="(1,0)";
(2,0)="(2,0)";
(3,0)="(3,0)";
(4,0)="(4,0)";
(5,0)="(5,0)";
(6,0)="(6,0)";
(7,0)="(7,0)";
(8,0)="(8,0)";
(9,0)="(9,0)";
(10,0)="(10,0)";
(11,0)="(11,0)";
(12,0)="(12,0)";
(13,0)="(13,0)";
(14,0)="(14,0)";
(1,14)="(1,14)";
(2,14)="(2,14)";
(3,14)="(3,14)";
(4,14)="(4,14)";
(5,14)="(5,14)";
(6,14)="(6,14)";
(7,14)="(7,14)";
(8,14)="(8,14)";
(9,14)="(9,14)";
(10,14)="(10,14)";
(11,14)="(11,14)";
(12,14)="(12,14)";
(13,14)="(13,14)";
(14,14)="(14,14)";
(0,1)="(0,1)";
(0,2)="(0,2)";
(0,3)="(0,3)";
(0,4)="(0,4)";
(0,5)="(0,5)";
(0,6)="(0,6)";
(0,7)="(0,7)";
(0,8)="(0,8)";
(0,9)="(0,9)";
(0,10)="(0,10)";
(0,11)="(0,11)";
(0,12)="(0,12)";
(0,13)="(0,13)";
(0,14)="(0,14)";
(14,1)="(14,1)";
(14,2)="(14,2)";
(14,3)="(14,3)";
(14,4)="(14,4)";
(14,5)="(14,5)";
(14,6)="(14,6)";
(14,7)="(14,7)";
(14,8)="(14,8)";
(14,9)="(14,9)";
(14,10)="(14,10)";
(14,11)="(14,11)";
(14,12)="(14,12)";
(14,13)="(14,13)";
(14,14)="(14,14)";
 "(0,0)"; "(14,0)"**\dir{.};
  "(0,1)"; "(14,1)"**\dir{.};
"(0,2)"; "(14,2)"**\dir{.};
"(0,3)"; "(14,3)"**\dir{.};
"(0,4)"; "(14,4)"**\dir{.};
"(0,5)"; "(14,5)"**\dir{.};
"(0,6)"; "(14,6)"**\dir{.};
"(0,7)"; "(14,7)"**\dir{.};
"(0,8)"; "(14,8)"**\dir{.};
  "(0,9)"; "(14,9)"**\dir{.};
"(0,10)"; "(14,10)"**\dir{.};
"(0,11)"; "(14,11)"**\dir{.};
"(0,12)"; "(14,12)"**\dir{.};
"(0,13)"; "(14,13)"**\dir{.};
"(0,14)"; "(14,14)"**\dir{.};
"(0,0)"; "(0,14)"**\dir{.};
"(1,0)"; "(1,14)"**\dir{.};
"(2,0)"; "(2,14)"**\dir{.};
"(3,0)"; "(3,14)"**\dir{.};
"(4,0)"; "(4,14)"**\dir{.};
"(5,0)"; "(5,14)"**\dir{.};
"(6,0)"; "(6,14)"**\dir{.};
"(7,0)"; "(7,14)"**\dir{.};
"(8,0)"; "(8,14)"**\dir{.};
"(9,0)"; "(9,14)"**\dir{.};
"(10,0)"; "(10,14)"**\dir{.};
"(11,0)"; "(11,14)"**\dir{.};
"(12,0)"; "(12,14)"**\dir{.};
"(13,0)"; "(13,14)"**\dir{.};
"(14,0)"; "(14,14)"**\dir{.};
(0.5,0)="(0.5,0)";
(1.5,0)="(1.5,0)";
(2.5,0)="(2.5,0)";
(3.5,0)="(3.5,0)";
(4.5,0)="(4.5,0)";
(5.5,0)="(5.5,0)";
(6.5,0)="(6.5,0)";
(7.5,0)="(7.5,0)";
(8.5,0)="(8.5,0)";
(9.5,0)="(9.5,0)";
(10.5,0)="(10.5,0)";
(11.5,0)="(11.5,0)";
(12.5,0)="(12.5,0)";
(13.5,0)="(13.5,0)";
(14,0.5)="(14,0.5)";
(14,1.5)="(14,1.5)";
(14,2.5)="(14,2.5)";
(14,3.5)="(14,3.5)";
(14,4.5)="(14,4.5)";
(14,5.5)="(14,5.5)";
(14,6.5)="(14,6.5)";
(14,7.5)="(14,7.5)";
(14,8.5)="(14,8.5)";
(14,9.5)="(14,9.5)";
(14,10.5)="(14,10.5)";
(14,11.5)="(14,11.5)";
(14,12.5)="(14,12.5)";
(14,13.5)="(14,13.5)";
(2.5,13.5)*{\circ}="w1";
(1.5,12.5)*{\circ}="w2";
(0.5,11.5)*{\circ}="w3";
(4.5,10.5)*{\circ}="w4";
(3.5,9.5)*{\circ}="w5";
(10.5,8.5)*{\circ}="w6";
(13.5,7.5)*{\circ}="w7";
(12.5,6.5)*{\circ}="w8";
(8.5,5.5)*{\circ}="w9";
(7.5,4.5)*{\circ}="w10";
(9.5,3.5)*{\circ}="w11";
(5.5,2.5)*{\circ}="w12";
(6.5,1.5)*{\circ}="w13";
(11.5,0.5)*{\circ}="w14";
"w1"; "(2.5,0)"**\dir{-};
"w1"; "(14,13.5)"**\dir{-};
"w2"; "(1.5,0)"**\dir{-};
"w2"; "(14,12.5)"**\dir{-};
"w3"; "(0.5,0)"**\dir{-};
"w3"; "(14,11.5)"**\dir{-};
"w4"; "(4.5,0)"**\dir{-};
"w4"; "(14,10.5)"**\dir{-};
"w5"; "(3.5,0)"**\dir{-};
"w5"; "(14,9.5)"**\dir{-};
"w6"; "(10.5,0)"**\dir{-};
"w6"; "(14,8.5)"**\dir{-};
"w7"; "(13.5,0)"**\dir{-};
"w7"; "(14,7.5)"**\dir{-};
"w8"; "(12.5,0)"**\dir{-};
"w8"; "(14,6.5)"**\dir{-};
"w9"; "(8.5,0)"**\dir{-};
"w9"; "(14,5.5)"**\dir{-};
"w10"; "(7.5,0)"**\dir{-};
"w10"; "(14,4.5)"**\dir{-};
"w11"; "(9.5,0)"**\dir{-};
"w11"; "(14,3.5)"**\dir{-};
"w12"; "(5.5,0)"**\dir{-};
"w12"; "(14,2.5)"**\dir{-};
"w13"; "(6.5,0)"**\dir{-};
"w13"; "(14,1.5)"**\dir{-};
"w14"; "(11.5,0)"**\dir{-};
"w14"; "(14,0.5)"**\dir{-};
\endxy
\,\,\,\,\,B_{11}:\,\, \xy 0;/r.5pc/:
(0,0)="(0,0)";
(1,0)="(1,0)";
(2,0)="(2,0)";
(3,0)="(3,0)";
(4,0)="(4,0)";
(5,0)="(5,0)";
(6,0)="(6,0)";
(7,0)="(7,0)";
(8,0)="(8,0)";
(9,0)="(9,0)";
(10,0)="(10,0)";
(11,0)="(11,0)";
(12,0)="(12,0)";
(13,0)="(13,0)";
(14,0)="(14,0)";
(15,0)="(15,0)";
(1,15)="(1,15)";
(2,15)="(2,15)";
(3,15)="(3,15)";
(4,15)="(4,15)";
(5,15)="(5,15)";
(6,15)="(6,15)";
(7,15)="(7,15)";
(8,15)="(8,15)";
(9,15)="(9,15)";
(10,15)="(10,15)";
(11,15)="(11,15)";
(12,15)="(12,15)";
(13,15)="(13,15)";
(14,15)="(14,15)";
(15,15)="(15,15)";
(0,1)="(0,1)";
(0,2)="(0,2)";
(0,3)="(0,3)";
(0,4)="(0,4)";
(0,5)="(0,5)";
(0,6)="(0,6)";
(0,7)="(0,7)";
(0,8)="(0,8)";
(0,9)="(0,9)";
(0,10)="(0,10)";
(0,11)="(0,11)";
(0,12)="(0,12)";
(0,13)="(0,13)";
(0,14)="(0,14)";
(0,15)="(0,15)";
(15,1)="(15,1)";
(15,2)="(15,2)";
(15,3)="(15,3)";
(15,4)="(15,4)";
(15,5)="(15,5)";
(15,6)="(15,6)";
(15,7)="(15,7)";
(15,8)="(15,8)";
(15,9)="(15,9)";
(15,10)="(15,10)";
(15,11)="(15,11)";
(15,12)="(15,12)";
(15,13)="(15,13)";
(15,14)="(15,14)";
 "(0,0)"; "(15,0)"**\dir{.};
  "(0,1)"; "(15,1)"**\dir{.};
"(0,2)"; "(15,2)"**\dir{.};
"(0,3)"; "(15,3)"**\dir{.};
"(0,4)"; "(15,4)"**\dir{.};
"(0,5)"; "(15,5)"**\dir{.};
"(0,6)"; "(15,6)"**\dir{.};
"(0,7)"; "(15,7)"**\dir{.};
"(0,8)"; "(15,8)"**\dir{.};
  "(0,9)"; "(15,9)"**\dir{.};
"(0,10)"; "(15,10)"**\dir{.};
"(0,11)"; "(15,11)"**\dir{.};
"(0,12)"; "(15,12)"**\dir{.};
"(0,13)"; "(15,13)"**\dir{.};
"(0,14)"; "(15,14)"**\dir{.};
"(0,15)"; "(15,15)"**\dir{.};
"(0,0)"; "(0,15)"**\dir{.};
"(1,0)"; "(1,15)"**\dir{.};
"(2,0)"; "(2,15)"**\dir{.};
"(3,0)"; "(3,15)"**\dir{.};
"(4,0)"; "(4,15)"**\dir{.};
"(5,0)"; "(5,15)"**\dir{.};
"(6,0)"; "(6,15)"**\dir{.};
"(7,0)"; "(7,15)"**\dir{.};
"(8,0)"; "(8,15)"**\dir{.};
"(9,0)"; "(9,15)"**\dir{.};
"(10,0)"; "(10,15)"**\dir{.};
"(11,0)"; "(11,15)"**\dir{.};
"(12,0)"; "(12,15)"**\dir{.};
"(13,0)"; "(13,15)"**\dir{.};
"(14,0)"; "(14,15)"**\dir{.};
"(15,0)"; "(15,15)"**\dir{.};
(0.5,0)="(0.5,0)";
(1.5,0)="(1.5,0)";
(2.5,0)="(2.5,0)";
(3.5,0)="(3.5,0)";
(4.5,0)="(4.5,0)";
(5.5,0)="(5.5,0)";
(6.5,0)="(6.5,0)";
(7.5,0)="(7.5,0)";
(8.5,0)="(8.5,0)";
(9.5,0)="(9.5,0)";
(10.5,0)="(10.5,0)";
(11.5,0)="(11.5,0)";
(12.5,0)="(12.5,0)";
(13.5,0)="(13.5,0)";
(14.5,0)="(14.5,0)";
(15,0.5)="(15,0.5)";
(15,1.5)="(15,1.5)";
(15,2.5)="(15,2.5)";
(15,3.5)="(15,3.5)";
(15,4.5)="(15,4.5)";
(15,5.5)="(15,5.5)";
(15,6.5)="(15,6.5)";
(15,7.5)="(15,7.5)";
(15,8.5)="(15,8.5)";
(15,9.5)="(15,9.5)";
(15,10.5)="(15,10.5)";
(15,11.5)="(15,11.5)";
(15,12.5)="(15,12.5)";
(15,13.5)="(15,13.5)";
(15,14.5)="(15,14.5)";
(2.5,14.5)*{\circ}="w1";
(1.5,13.5)*{\circ}="w2";
(0.5,12.5)*{\circ}="w3";
(4.5,11.5)*{\circ}="w4";
(3.5,10.5)*{\circ}="w5";
(5.5,9.5)*{\times}="w6";
(11.5,8.5)*{\circ}="w7";
(14.5,7.5)*{\circ}="w8";
(13.5,6.5)*{\circ}="w9";
(9.5,5.5)*{\circ}="w10";
(8.5,4.5)*{\circ}="w11";
(10.5,3.5)*{\circ}="w12";
(6.5,2.5)*{\circ}="w13";
(7.5,1.5)*{\circ}="w14";
(12.5,0.5)*{\circ}="w15";
"w1"; "(2.5,0)"**\dir{-};
"w1"; "(15,14.5)"**\dir{-};
"w2"; "(1.5,0)"**\dir{-};
"w2"; "(15,13.5)"**\dir{-};
"w3"; "(0.5,0)"**\dir{-};
"w3"; "(15,12.5)"**\dir{-};
"w4"; "(4.5,0)"**\dir{-};
"w4"; "(15,11.5)"**\dir{-};
"w5"; "(3.5,0)"**\dir{-};
"w5"; "(15,10.5)"**\dir{-};
"w6"; "(5.5,0)"**\dir{-};
"w6"; "(15,9.5)"**\dir{-};
"w7"; "(11.5,0)"**\dir{-};
"w7"; "(15,8.5)"**\dir{-};
"w8"; "(14.5,0)"**\dir{-};
"w8"; "(15,7.5)"**\dir{-};
"w9"; "(13.5,0)"**\dir{-};
"w9"; "(15,6.5)"**\dir{-};
"w10"; "(9.5,0)"**\dir{-};
"w10"; "(15,5.5)"**\dir{-};
"w11"; "(8.5,0)"**\dir{-};
"w11"; "(15,4.5)"**\dir{-};
"w12"; "(10.5,0)"**\dir{-};
"w12"; "(15,3.5)"**\dir{-};
"w13"; "(6.5,0)"**\dir{-};
"w13"; "(15,2.5)"**\dir{-};
"w14"; "(7.5,0)"**\dir{-};
"w14"; "(15,1.5)"**\dir{-};
"w15"; "(12.5,0)"**\dir{-};
"w15"; "(15,0.5)"**\dir{-};
\endxy$$
Then we apply MT-moves to both of them. Here is part of the tree. 
$$
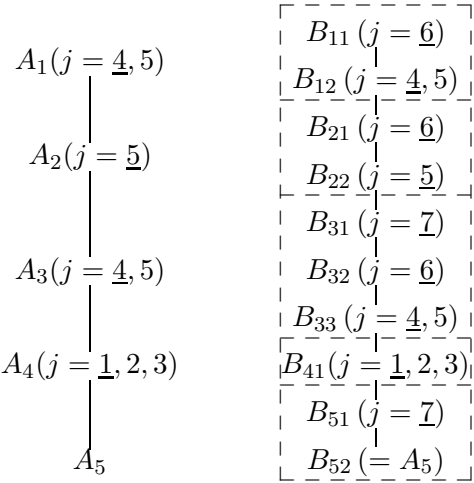
\xy 0;/r.30pc/:
(-20,0)*{A_5}="a5";
(-20,10)*{A_4 (j=\underline{1},2,3)}="a4";
(-20,20)*{A_3 (j=\underline{4},5)}="a3";
(-20,32)*{A_2 (j=\underline{5})}="a2";
(-20,42)*{A_1 (j=\underline{4},5)}="a1";
 "a5"; "a4"**\dir{-};
 "a3"; "a4"**\dir{-};
  "a3"; "a2"**\dir{-};
   "a1"; "a2"**\dir{-};
(10,0)*{B_{52}\, (=A_5)}="b52";
(10,5)*{B_{51}\, (j=\underline{7})}="b51";
(10,10)*{B_{41} (j=\underline{1},2,3)}="b41";
(10,15)*{B_{33}\, (j=\underline{4},5)}="b33";
(10,20)*{B_{32} \,(j=\underline{6})}="b32";
(10,25)*{B_{31}\, (j=\underline{7})}="b31";
(10,30)*{B_{22}\, (j=\underline{5})}="b22";
(10,35)*{B_{21} \,(j=\underline{6})}="b21";
(10,40)*{B_{12}\, (j=\underline{4},5)}="b12";
(10,45)*{B_{11}\, (j=\underline{6})}="b11";
 "b11"; "b12"**\dir{-};
"b12"; "b21"**\dir{-};
"b22"; "b21"**\dir{-};
"b22"; "b31"**\dir{-};
"b32"; "b31"**\dir{-};
"b32"; "b33"**\dir{-};
"b41"; "b33"**\dir{-};
"b41"; "b51"**\dir{-};
"b52"; "b51"**\dir{-};
(0,-2)="1";(20,-2)="11";"1"; "11"**\dir{--};
(0,8)="2";(20,8)="22";"2"; "22"**\dir{--};
(0,13)="3";(20,13)="33";"3"; "33"**\dir{--};
(0,28)="4";(20,28)="44";"4"; "44"**\dir{--};
(0,38)="5";(20,38)="55";"5"; "55"**\dir{--};
(0,48)="6";(20,48)="66";"6"; "66"**\dir{--};
"6"; "1"**\dir{--};
"11"; "66"**\dir{--};
\endxy
$$
Even though the path of the $B_i$'s is longer, but eventually, it gets to the same permutation $B_{52}=A_5$. For example, consider $A_3$ and $B_{31}$. In $A_3$, the corner box is indicated by a bullet. While in $B_{31}$, the coresponding box is also indicated by a bullet. But it is not the corner box in $B_{31}$, since there are some delayed boxes to its right (indicated by the dots). So we need to do some cleaning (the move from $B_{31}$ to $B_{32}$, and then to $B_{33}$) before moving the same block as $A_3$. Finally, in $B_{33}$, we are ready to move the same block as $A_3$.
$$A_3:\,\,
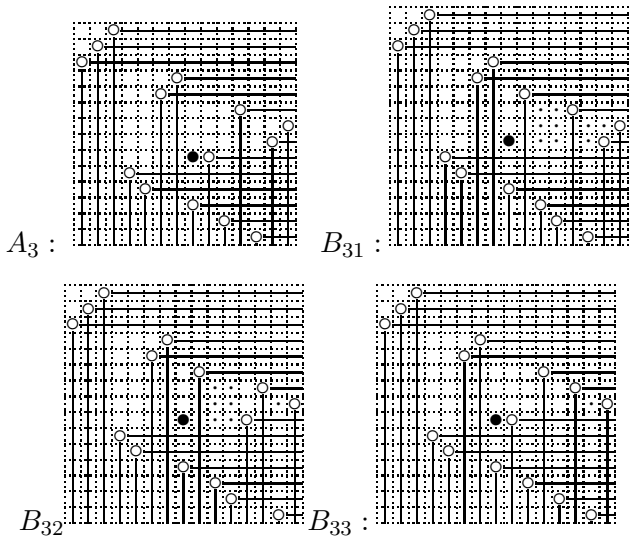
\xy 0;/r.5pc/:
(0,0)="(0,0)";
(1,0)="(1,0)";
(2,0)="(2,0)";
(3,0)="(3,0)";
(4,0)="(4,0)";
(5,0)="(5,0)";
(6,0)="(6,0)";
(7,0)="(7,0)";
(8,0)="(8,0)";
(9,0)="(9,0)";
(10,0)="(10,0)";
(11,0)="(11,0)";
(12,0)="(12,0)";
(13,0)="(13,0)";
(14,0)="(14,0)";
(1,14)="(1,14)";
(2,14)="(2,14)";
(3,14)="(3,14)";
(4,14)="(4,14)";
(5,14)="(5,14)";
(6,14)="(6,14)";
(7,14)="(7,14)";
(8,14)="(8,14)";
(9,14)="(9,14)";
(10,14)="(10,14)";
(11,14)="(11,14)";
(12,14)="(12,14)";
(13,14)="(13,14)";
(14,14)="(14,14)";
(0,1)="(0,1)";
(0,2)="(0,2)";
(0,3)="(0,3)";
(0,4)="(0,4)";
(0,5)="(0,5)";
(0,6)="(0,6)";
(0,7)="(0,7)";
(0,8)="(0,8)";
(0,9)="(0,9)";
(0,10)="(0,10)";
(0,11)="(0,11)";
(0,12)="(0,12)";
(0,13)="(0,13)";
(0,14)="(0,14)";
(14,1)="(14,1)";
(14,2)="(14,2)";
(14,3)="(14,3)";
(14,4)="(14,4)";
(14,5)="(14,5)";
(14,6)="(14,6)";
(14,7)="(14,7)";
(14,8)="(14,8)";
(14,9)="(14,9)";
(14,10)="(14,10)";
(14,11)="(14,11)";
(14,12)="(14,12)";
(14,13)="(14,13)";
(14,14)="(14,14)";
 "(0,0)"; "(14,0)"**\dir{.};
  "(0,1)"; "(14,1)"**\dir{.};
"(0,2)"; "(14,2)"**\dir{.};
"(0,3)"; "(14,3)"**\dir{.};
"(0,4)"; "(14,4)"**\dir{.};
"(0,5)"; "(14,5)"**\dir{.};
"(0,6)"; "(14,6)"**\dir{.};
"(0,7)"; "(14,7)"**\dir{.};
"(0,8)"; "(14,8)"**\dir{.};
  "(0,9)"; "(14,9)"**\dir{.};
"(0,10)"; "(14,10)"**\dir{.};
"(0,11)"; "(14,11)"**\dir{.};
"(0,12)"; "(14,12)"**\dir{.};
"(0,13)"; "(14,13)"**\dir{.};
"(0,14)"; "(14,14)"**\dir{.};
"(0,0)"; "(0,14)"**\dir{.};
"(1,0)"; "(1,14)"**\dir{.};
"(2,0)"; "(2,14)"**\dir{.};
"(3,0)"; "(3,14)"**\dir{.};
"(4,0)"; "(4,14)"**\dir{.};
"(5,0)"; "(5,14)"**\dir{.};
"(6,0)"; "(6,14)"**\dir{.};
"(7,0)"; "(7,14)"**\dir{.};
"(8,0)"; "(8,14)"**\dir{.};
"(9,0)"; "(9,14)"**\dir{.};
"(10,0)"; "(10,14)"**\dir{.};
"(11,0)"; "(11,14)"**\dir{.};
"(12,0)"; "(12,14)"**\dir{.};
"(13,0)"; "(13,14)"**\dir{.};
"(14,0)"; "(14,14)"**\dir{.};
(0.5,0)="(0.5,0)";
(1.5,0)="(1.5,0)";
(2.5,0)="(2.5,0)";
(3.5,0)="(3.5,0)";
(4.5,0)="(4.5,0)";
(5.5,0)="(5.5,0)";
(6.5,0)="(6.5,0)";
(7.5,0)="(7.5,0)";
(8.5,0)="(8.5,0)";
(9.5,0)="(9.5,0)";
(10.5,0)="(10.5,0)";
(11.5,0)="(11.5,0)";
(12.5,0)="(12.5,0)";
(13.5,0)="(13.5,0)";
(14,0.5)="(14,0.5)";
(14,1.5)="(14,1.5)";
(14,2.5)="(14,2.5)";
(14,3.5)="(14,3.5)";
(14,4.5)="(14,4.5)";
(14,5.5)="(14,5.5)";
(14,6.5)="(14,6.5)";
(14,7.5)="(14,7.5)";
(14,8.5)="(14,8.5)";
(14,9.5)="(14,9.5)";
(14,10.5)="(14,10.5)";
(14,11.5)="(14,11.5)";
(14,12.5)="(14,12.5)";
(14,13.5)="(14,13.5)";
(2.5,13.5)*{\circ}="w1";
(1.5,12.5)*{\circ}="w2";
(0.5,11.5)*{\circ}="w3";
(6.5,10.5)*{\circ}="w4";
(5.5,9.5)*{\circ}="w5";
(10.5,8.5)*{\circ}="w6";
(13.5,7.5)*{\circ}="w7";
(12.5,6.5)*{\circ}="w8";
(8.5,5.5)*{\circ}="w9";
(3.5,4.5)*{\circ}="w10";
(4.5,3.5)*{\circ}="w11";
(7.5,2.5)*{\circ}="w12";
(9.5,1.5)*{\circ}="w13";
(11.5,0.5)*{\circ}="w14";
"w1"; "(2.5,0)"**\dir{-};
"w1"; "(14,13.5)"**\dir{-};
"w2"; "(1.5,0)"**\dir{-};
"w2"; "(14,12.5)"**\dir{-};
"w3"; "(0.5,0)"**\dir{-};
"w3"; "(14,11.5)"**\dir{-};
"w4"; "(6.5,0)"**\dir{-};
"w4"; "(14,10.5)"**\dir{-};
"w5"; "(5.5,0)"**\dir{-};
"w5"; "(14,9.5)"**\dir{-};
"w6"; "(10.5,0)"**\dir{-};
"w6"; "(14,8.5)"**\dir{-};
"w7"; "(13.5,0)"**\dir{-};
"w7"; "(14,7.5)"**\dir{-};
"w8"; "(12.5,0)"**\dir{-};
"w8"; "(14,6.5)"**\dir{-};
"w9"; "(8.5,0)"**\dir{-};
"w9"; "(14,5.5)"**\dir{-};
"w10"; "(3.5,0)"**\dir{-};
"w10"; "(14,4.5)"**\dir{-};
"w11"; "(4.5,0)"**\dir{-};
"w11"; "(14,3.5)"**\dir{-};
"w12"; "(7.5,0)"**\dir{-};
"w12"; "(14,2.5)"**\dir{-};
"w13"; "(9.5,0)"**\dir{-};
"w13"; "(14,1.5)"**\dir{-};
"w14"; "(11.5,0)"**\dir{-};
"w14"; "(14,0.5)"**\dir{-};
(7.5,5.5)*{\bullet};
\endxy
\,\,\,  \,\, B_{31}:
\xy 0;/r.5pc/:
(0,0)="(0,0)";
(1,0)="(1,0)";
(2,0)="(2,0)";
(3,0)="(3,0)";
(4,0)="(4,0)";
(5,0)="(5,0)";
(6,0)="(6,0)";
(7,0)="(7,0)";
(8,0)="(8,0)";
(9,0)="(9,0)";
(10,0)="(10,0)";
(11,0)="(11,0)";
(12,0)="(12,0)";
(13,0)="(13,0)";
(14,0)="(14,0)";
(15,0)="(15,0)";
(1,15)="(1,15)";
(2,15)="(2,15)";
(3,15)="(3,15)";
(4,15)="(4,15)";
(5,15)="(5,15)";
(6,15)="(6,15)";
(7,15)="(7,15)";
(8,15)="(8,15)";
(9,15)="(9,15)";
(10,15)="(10,15)";
(11,15)="(11,15)";
(12,15)="(12,15)";
(13,15)="(13,15)";
(14,15)="(14,15)";
(15,15)="(15,15)";
(0,1)="(0,1)";
(0,2)="(0,2)";
(0,3)="(0,3)";
(0,4)="(0,4)";
(0,5)="(0,5)";
(0,6)="(0,6)";
(0,7)="(0,7)";
(0,8)="(0,8)";
(0,9)="(0,9)";
(0,10)="(0,10)";
(0,11)="(0,11)";
(0,12)="(0,12)";
(0,13)="(0,13)";
(0,14)="(0,14)";
(0,15)="(0,15)";
(15,1)="(15,1)";
(15,2)="(15,2)";
(15,3)="(15,3)";
(15,4)="(15,4)";
(15,5)="(15,5)";
(15,6)="(15,6)";
(15,7)="(15,7)";
(15,8)="(15,8)";
(15,9)="(15,9)";
(15,10)="(15,10)";
(15,11)="(15,11)";
(15,12)="(15,12)";
(15,13)="(15,13)";
(15,14)="(15,14)";
 "(0,0)"; "(15,0)"**\dir{.};
  "(0,1)"; "(15,1)"**\dir{.};
"(0,2)"; "(15,2)"**\dir{.};
"(0,3)"; "(15,3)"**\dir{.};
"(0,4)"; "(15,4)"**\dir{.};
"(0,5)"; "(15,5)"**\dir{.};
"(0,6)"; "(15,6)"**\dir{.};
"(0,7)"; "(15,7)"**\dir{.};
"(0,8)"; "(15,8)"**\dir{.};
  "(0,9)"; "(15,9)"**\dir{.};
"(0,10)"; "(15,10)"**\dir{.};
"(0,11)"; "(15,11)"**\dir{.};
"(0,12)"; "(15,12)"**\dir{.};
"(0,13)"; "(15,13)"**\dir{.};
"(0,14)"; "(15,14)"**\dir{.};
"(0,15)"; "(15,15)"**\dir{.};
"(0,0)"; "(0,15)"**\dir{.};
"(1,0)"; "(1,15)"**\dir{.};
"(2,0)"; "(2,15)"**\dir{.};
"(3,0)"; "(3,15)"**\dir{.};
"(4,0)"; "(4,15)"**\dir{.};
"(5,0)"; "(5,15)"**\dir{.};
"(6,0)"; "(6,15)"**\dir{.};
"(7,0)"; "(7,15)"**\dir{.};
"(8,0)"; "(8,15)"**\dir{.};
"(9,0)"; "(9,15)"**\dir{.};
"(10,0)"; "(10,15)"**\dir{.};
"(11,0)"; "(11,15)"**\dir{.};
"(12,0)"; "(12,15)"**\dir{.};
"(13,0)"; "(13,15)"**\dir{.};
"(14,0)"; "(14,15)"**\dir{.};
"(15,0)"; "(15,15)"**\dir{.};
(0.5,0)="(0.5,0)";
(1.5,0)="(1.5,0)";
(2.5,0)="(2.5,0)";
(3.5,0)="(3.5,0)";
(4.5,0)="(4.5,0)";
(5.5,0)="(5.5,0)";
(6.5,0)="(6.5,0)";
(7.5,0)="(7.5,0)";
(8.5,0)="(8.5,0)";
(9.5,0)="(9.5,0)";
(10.5,0)="(10.5,0)";
(11.5,0)="(11.5,0)";
(12.5,0)="(12.5,0)";
(13.5,0)="(13.5,0)";
(14.5,0)="(14.5,0)";
(15,0.5)="(15,0.5)";
(15,1.5)="(15,1.5)";
(15,2.5)="(15,2.5)";
(15,3.5)="(15,3.5)";
(15,4.5)="(15,4.5)";
(15,5.5)="(15,5.5)";
(15,6.5)="(15,6.5)";
(15,7.5)="(15,7.5)";
(15,8.5)="(15,8.5)";
(15,9.5)="(15,9.5)";
(15,10.5)="(15,10.5)";
(15,11.5)="(15,11.5)";
(15,12.5)="(15,12.5)";
(15,13.5)="(15,13.5)";
(15,14.5)="(15,14.5)";
(2.5,14.5)*{\circ}="w1";
(1.5,13.5)*{\circ}="w2";
(0.5,12.5)*{\circ}="w3";
(6.5,11.5)*{\circ}="w4";
(5.5,10.5)*{\circ}="w5";
(8.5,9.5)*{\circ}="w6";
(11.5,8.5)*{\circ}="w7";
(14.5,7.5)*{\circ}="w8";
(13.5,6.5)*{\circ}="w9";
(3.5,5.5)*{\circ}="w10";
(4.5,4.5)*{\circ}="w11";
(7.5,3.5)*{\circ}="w12";
(9.5,2.5)*{\circ}="w13";
(10.5,1.5)*{\circ}="w14";
(12.5,0.5)*{\circ}="w15";
"w1"; "(2.5,0)"**\dir{-};
"w1"; "(15,14.5)"**\dir{-};
"w2"; "(1.5,0)"**\dir{-};
"w2"; "(15,13.5)"**\dir{-};
"w3"; "(0.5,0)"**\dir{-};
"w3"; "(15,12.5)"**\dir{-};
"w4"; "(6.5,0)"**\dir{-};
"w4"; "(15,11.5)"**\dir{-};
"w5"; "(5.5,0)"**\dir{-};
"w5"; "(15,10.5)"**\dir{-};
"w6"; "(8.5,0)"**\dir{-};
"w6"; "(15,9.5)"**\dir{-};
"w7"; "(11.5,0)"**\dir{-};
"w7"; "(15,8.5)"**\dir{-};
"w8"; "(14.5,0)"**\dir{-};
"w8"; "(15,7.5)"**\dir{-};
"w9"; "(13.5,0)"**\dir{-};
"w9"; "(15,6.5)"**\dir{-};
"w10"; "(3.5,0)"**\dir{-};
"w10"; "(15,5.5)"**\dir{-};
"w11"; "(4.5,0)"**\dir{-};
"w11"; "(15,4.5)"**\dir{-};
"w12"; "(7.5,0)"**\dir{-};
"w12"; "(15,3.5)"**\dir{-};
"w13"; "(9.5,0)"**\dir{-};
"w13"; "(15,2.5)"**\dir{-};
"w14"; "(10.5,0)"**\dir{-};
"w14"; "(15,1.5)"**\dir{-};
"w15"; "(12.5,0)"**\dir{-};
"w15"; "(15,0.5)"**\dir{-};
(7.5,6.5)*{\bullet};
(9.5,6.5)*{\cdot};(10.5,6.5)*{\cdot};(12.5,6.5)*{\cdot};
(9.5,7.5)*{\cdot};(10.5,7.5)*{\cdot};(12.5,7.5)*{\cdot};(13.5,7.5)*{\cdot};
(9.5,8.5)*{\cdot};(10.5,8.5)*{\cdot};
\endxy$$
$$B_{32} \xy 0;/r.5pc/:
(0,0)="(0,0)";
(1,0)="(1,0)";
(2,0)="(2,0)";
(3,0)="(3,0)";
(4,0)="(4,0)";
(5,0)="(5,0)";
(6,0)="(6,0)";
(7,0)="(7,0)";
(8,0)="(8,0)";
(9,0)="(9,0)";
(10,0)="(10,0)";
(11,0)="(11,0)";
(12,0)="(12,0)";
(13,0)="(13,0)";
(14,0)="(14,0)";
(15,0)="(15,0)";
(1,15)="(1,15)";
(2,15)="(2,15)";
(3,15)="(3,15)";
(4,15)="(4,15)";
(5,15)="(5,15)";
(6,15)="(6,15)";
(7,15)="(7,15)";
(8,15)="(8,15)";
(9,15)="(9,15)";
(10,15)="(10,15)";
(11,15)="(11,15)";
(12,15)="(12,15)";
(13,15)="(13,15)";
(14,15)="(14,15)";
(15,15)="(15,15)";
(0,1)="(0,1)";
(0,2)="(0,2)";
(0,3)="(0,3)";
(0,4)="(0,4)";
(0,5)="(0,5)";
(0,6)="(0,6)";
(0,7)="(0,7)";
(0,8)="(0,8)";
(0,9)="(0,9)";
(0,10)="(0,10)";
(0,11)="(0,11)";
(0,12)="(0,12)";
(0,13)="(0,13)";
(0,14)="(0,14)";
(0,15)="(0,15)";
(15,1)="(15,1)";
(15,2)="(15,2)";
(15,3)="(15,3)";
(15,4)="(15,4)";
(15,5)="(15,5)";
(15,6)="(15,6)";
(15,7)="(15,7)";
(15,8)="(15,8)";
(15,9)="(15,9)";
(15,10)="(15,10)";
(15,11)="(15,11)";
(15,12)="(15,12)";
(15,13)="(15,13)";
(15,14)="(15,14)";
 "(0,0)"; "(15,0)"**\dir{.};
  "(0,1)"; "(15,1)"**\dir{.};
"(0,2)"; "(15,2)"**\dir{.};
"(0,3)"; "(15,3)"**\dir{.};
"(0,4)"; "(15,4)"**\dir{.};
"(0,5)"; "(15,5)"**\dir{.};
"(0,6)"; "(15,6)"**\dir{.};
"(0,7)"; "(15,7)"**\dir{.};
"(0,8)"; "(15,8)"**\dir{.};
  "(0,9)"; "(15,9)"**\dir{.};
"(0,10)"; "(15,10)"**\dir{.};
"(0,11)"; "(15,11)"**\dir{.};
"(0,12)"; "(15,12)"**\dir{.};
"(0,13)"; "(15,13)"**\dir{.};
"(0,14)"; "(15,14)"**\dir{.};
"(0,15)"; "(15,15)"**\dir{.};
"(0,0)"; "(0,15)"**\dir{.};
"(1,0)"; "(1,15)"**\dir{.};
"(2,0)"; "(2,15)"**\dir{.};
"(3,0)"; "(3,15)"**\dir{.};
"(4,0)"; "(4,15)"**\dir{.};
"(5,0)"; "(5,15)"**\dir{.};
"(6,0)"; "(6,15)"**\dir{.};
"(7,0)"; "(7,15)"**\dir{.};
"(8,0)"; "(8,15)"**\dir{.};
"(9,0)"; "(9,15)"**\dir{.};
"(10,0)"; "(10,15)"**\dir{.};
"(11,0)"; "(11,15)"**\dir{.};
"(12,0)"; "(12,15)"**\dir{.};
"(13,0)"; "(13,15)"**\dir{.};
"(14,0)"; "(14,15)"**\dir{.};
"(15,0)"; "(15,15)"**\dir{.};
(0.5,0)="(0.5,0)";
(1.5,0)="(1.5,0)";
(2.5,0)="(2.5,0)";
(3.5,0)="(3.5,0)";
(4.5,0)="(4.5,0)";
(5.5,0)="(5.5,0)";
(6.5,0)="(6.5,0)";
(7.5,0)="(7.5,0)";
(8.5,0)="(8.5,0)";
(9.5,0)="(9.5,0)";
(10.5,0)="(10.5,0)";
(11.5,0)="(11.5,0)";
(12.5,0)="(12.5,0)";
(13.5,0)="(13.5,0)";
(14.5,0)="(14.5,0)";
(15,0.5)="(15,0.5)";
(15,1.5)="(15,1.5)";
(15,2.5)="(15,2.5)";
(15,3.5)="(15,3.5)";
(15,4.5)="(15,4.5)";
(15,5.5)="(15,5.5)";
(15,6.5)="(15,6.5)";
(15,7.5)="(15,7.5)";
(15,8.5)="(15,8.5)";
(15,9.5)="(15,9.5)";
(15,10.5)="(15,10.5)";
(15,11.5)="(15,11.5)";
(15,12.5)="(15,12.5)";
(15,13.5)="(15,13.5)";
(15,14.5)="(15,14.5)";
(2.5,14.5)*{\circ}="w1";
(1.5,13.5)*{\circ}="w2";
(0.5,12.5)*{\circ}="w3";
(6.5,11.5)*{\circ}="w4";
(5.5,10.5)*{\circ}="w5";
(8.5,9.5)*{\circ}="w6";
(12.5,8.5)*{\circ}="w7";
(14.5,7.5)*{\circ}="w8";
(11.5,6.5)*{\circ}="w9";
(3.5,5.5)*{\circ}="w10";
(4.5,4.5)*{\circ}="w11";
(7.5,3.5)*{\circ}="w12";
(9.5,2.5)*{\circ}="w13";
(10.5,1.5)*{\circ}="w14";
(13.5,0.5)*{\circ}="w15";
"w1"; "(2.5,0)"**\dir{-};
"w1"; "(15,14.5)"**\dir{-};
"w2"; "(1.5,0)"**\dir{-};
"w2"; "(15,13.5)"**\dir{-};
"w3"; "(0.5,0)"**\dir{-};
"w3"; "(15,12.5)"**\dir{-};
"w4"; "(6.5,0)"**\dir{-};
"w4"; "(15,11.5)"**\dir{-};
"w5"; "(5.5,0)"**\dir{-};
"w5"; "(15,10.5)"**\dir{-};
"w6"; "(8.5,0)"**\dir{-};
"w6"; "(15,9.5)"**\dir{-};
"w7"; "(12.5,0)"**\dir{-};
"w7"; "(15,8.5)"**\dir{-};
"w8"; "(14.5,0)"**\dir{-};
"w8"; "(15,7.5)"**\dir{-};
"w9"; "(11.5,0)"**\dir{-};
"w9"; "(15,6.5)"**\dir{-};
"w10"; "(3.5,0)"**\dir{-};
"w10"; "(15,5.5)"**\dir{-};
"w11"; "(4.5,0)"**\dir{-};
"w11"; "(15,4.5)"**\dir{-};
"w12"; "(7.5,0)"**\dir{-};
"w12"; "(15,3.5)"**\dir{-};
"w13"; "(9.5,0)"**\dir{-};
"w13"; "(15,2.5)"**\dir{-};
"w14"; "(10.5,0)"**\dir{-};
"w14"; "(15,1.5)"**\dir{-};
"w15"; "(13.5,0)"**\dir{-};
"w15"; "(15,0.5)"**\dir{-};
(7.5,6.5)*{\bullet};
(9.5,6.5)*{\cdot};(10.5,6.5)*{\cdot};
(9.5,7.5)*{\cdot};(10.5,7.5)*{\cdot};(13.5,7.5)*{\cdot};
(9.5,8.5)*{\cdot};(10.5,8.5)*{\cdot};
\endxy
\, B_{33}: \xy 0;/r.5pc/:
(0,0)="(0,0)";
(1,0)="(1,0)";
(2,0)="(2,0)";
(3,0)="(3,0)";
(4,0)="(4,0)";
(5,0)="(5,0)";
(6,0)="(6,0)";
(7,0)="(7,0)";
(8,0)="(8,0)";
(9,0)="(9,0)";
(10,0)="(10,0)";
(11,0)="(11,0)";
(12,0)="(12,0)";
(13,0)="(13,0)";
(14,0)="(14,0)";
(15,0)="(15,0)";
(1,15)="(1,15)";
(2,15)="(2,15)";
(3,15)="(3,15)";
(4,15)="(4,15)";
(5,15)="(5,15)";
(6,15)="(6,15)";
(7,15)="(7,15)";
(8,15)="(8,15)";
(9,15)="(9,15)";
(10,15)="(10,15)";
(11,15)="(11,15)";
(12,15)="(12,15)";
(13,15)="(13,15)";
(14,15)="(14,15)";
(15,15)="(15,15)";
(0,1)="(0,1)";
(0,2)="(0,2)";
(0,3)="(0,3)";
(0,4)="(0,4)";
(0,5)="(0,5)";
(0,6)="(0,6)";
(0,7)="(0,7)";
(0,8)="(0,8)";
(0,9)="(0,9)";
(0,10)="(0,10)";
(0,11)="(0,11)";
(0,12)="(0,12)";
(0,13)="(0,13)";
(0,14)="(0,14)";
(0,15)="(0,15)";
(15,1)="(15,1)";
(15,2)="(15,2)";
(15,3)="(15,3)";
(15,4)="(15,4)";
(15,5)="(15,5)";
(15,6)="(15,6)";
(15,7)="(15,7)";
(15,8)="(15,8)";
(15,9)="(15,9)";
(15,10)="(15,10)";
(15,11)="(15,11)";
(15,12)="(15,12)";
(15,13)="(15,13)";
(15,14)="(15,14)";
 "(0,0)"; "(15,0)"**\dir{.};
  "(0,1)"; "(15,1)"**\dir{.};
"(0,2)"; "(15,2)"**\dir{.};
"(0,3)"; "(15,3)"**\dir{.};
"(0,4)"; "(15,4)"**\dir{.};
"(0,5)"; "(15,5)"**\dir{.};
"(0,6)"; "(15,6)"**\dir{.};
"(0,7)"; "(15,7)"**\dir{.};
"(0,8)"; "(15,8)"**\dir{.};
  "(0,9)"; "(15,9)"**\dir{.};
"(0,10)"; "(15,10)"**\dir{.};
"(0,11)"; "(15,11)"**\dir{.};
"(0,12)"; "(15,12)"**\dir{.};
"(0,13)"; "(15,13)"**\dir{.};
"(0,14)"; "(15,14)"**\dir{.};
"(0,15)"; "(15,15)"**\dir{.};
"(0,0)"; "(0,15)"**\dir{.};
"(1,0)"; "(1,15)"**\dir{.};
"(2,0)"; "(2,15)"**\dir{.};
"(3,0)"; "(3,15)"**\dir{.};
"(4,0)"; "(4,15)"**\dir{.};
"(5,0)"; "(5,15)"**\dir{.};
"(6,0)"; "(6,15)"**\dir{.};
"(7,0)"; "(7,15)"**\dir{.};
"(8,0)"; "(8,15)"**\dir{.};
"(9,0)"; "(9,15)"**\dir{.};
"(10,0)"; "(10,15)"**\dir{.};
"(11,0)"; "(11,15)"**\dir{.};
"(12,0)"; "(12,15)"**\dir{.};
"(13,0)"; "(13,15)"**\dir{.};
"(14,0)"; "(14,15)"**\dir{.};
"(15,0)"; "(15,15)"**\dir{.};
(0.5,0)="(0.5,0)";
(1.5,0)="(1.5,0)";
(2.5,0)="(2.5,0)";
(3.5,0)="(3.5,0)";
(4.5,0)="(4.5,0)";
(5.5,0)="(5.5,0)";
(6.5,0)="(6.5,0)";
(7.5,0)="(7.5,0)";
(8.5,0)="(8.5,0)";
(9.5,0)="(9.5,0)";
(10.5,0)="(10.5,0)";
(11.5,0)="(11.5,0)";
(12.5,0)="(12.5,0)";
(13.5,0)="(13.5,0)";
(14.5,0)="(14.5,0)";
(15,0.5)="(15,0.5)";
(15,1.5)="(15,1.5)";
(15,2.5)="(15,2.5)";
(15,3.5)="(15,3.5)";
(15,4.5)="(15,4.5)";
(15,5.5)="(15,5.5)";
(15,6.5)="(15,6.5)";
(15,7.5)="(15,7.5)";
(15,8.5)="(15,8.5)";
(15,9.5)="(15,9.5)";
(15,10.5)="(15,10.5)";
(15,11.5)="(15,11.5)";
(15,12.5)="(15,12.5)";
(15,13.5)="(15,13.5)";
(15,14.5)="(15,14.5)";
(2.5,14.5)*{\circ}="w1";
(1.5,13.5)*{\circ}="w2";
(0.5,12.5)*{\circ}="w3";
(6.5,11.5)*{\circ}="w4";
(5.5,10.5)*{\circ}="w5";
(10.5,9.5)*{\circ}="w6";
(12.5,8.5)*{\circ}="w7";
(14.5,7.5)*{\circ}="w8";
(8.5,6.5)*{\circ}="w9";
(3.5,5.5)*{\circ}="w10";
(4.5,4.5)*{\circ}="w11";
(7.5,3.5)*{\circ}="w12";
(9.5,2.5)*{\circ}="w13";
(11.5,1.5)*{\circ}="w14";
(13.5,0.5)*{\circ}="w15";
"w1"; "(2.5,0)"**\dir{-};
"w1"; "(15,14.5)"**\dir{-};
"w2"; "(1.5,0)"**\dir{-};
"w2"; "(15,13.5)"**\dir{-};
"w3"; "(0.5,0)"**\dir{-};
"w3"; "(15,12.5)"**\dir{-};
"w4"; "(6.5,0)"**\dir{-};
"w4"; "(15,11.5)"**\dir{-};
"w5"; "(5.5,0)"**\dir{-};
"w5"; "(15,10.5)"**\dir{-};
"w6"; "(10.5,0)"**\dir{-};
"w6"; "(15,9.5)"**\dir{-};
"w7"; "(12.5,0)"**\dir{-};
"w7"; "(15,8.5)"**\dir{-};
"w8"; "(14.5,0)"**\dir{-};
"w8"; "(15,7.5)"**\dir{-};
"w9"; "(8.5,0)"**\dir{-};
"w9"; "(15,6.5)"**\dir{-};
"w10"; "(3.5,0)"**\dir{-};
"w10"; "(15,5.5)"**\dir{-};
"w11"; "(4.5,0)"**\dir{-};
"w11"; "(15,4.5)"**\dir{-};
"w12"; "(7.5,0)"**\dir{-};
"w12"; "(15,3.5)"**\dir{-};
"w13"; "(9.5,0)"**\dir{-};
"w13"; "(15,2.5)"**\dir{-};
"w14"; "(11.5,0)"**\dir{-};
"w14"; "(15,1.5)"**\dir{-};
"w15"; "(13.5,0)"**\dir{-};
"w15"; "(15,0.5)"**\dir{-};
(7.5,6.5)*{\bullet};
(13.5,7.5)*{\cdot};
\endxy$$

\end{example}

\begin{proof}[proof of Lemma \ref{gap}] Notice that $v^{-1}(1)-1$ is the number of boxes in the first column of $D(v)$. Consider again $w$ and $S(w)$ shown in Figure \ref{branching}.  Notice that applying different $j\in J(w)$ may result in different numbers $b(j)$ of potential boxes to be added to the first column. For example, for $j=2$, there is one box left, and for $j=3$, there are two boxes left (and are already added). The set $b=\{b(j)\mid j\in J(w)\}=[1,2]$ is an interval without any gaps. Using the diagram interpretation of the MT-move we can show that this holds in general, which implies this lemma.
\end{proof}
\begin{figure}[htp]
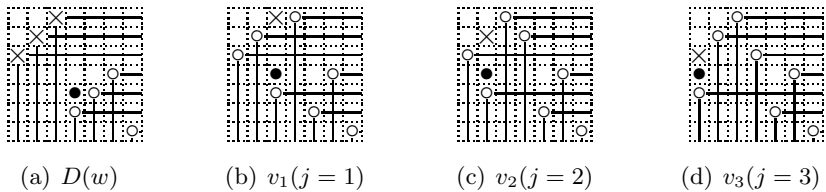

  \begin{center}
    \subfigure[$D(w)$]{\label{anc}$\,\,\,\,\,\,\,$ $\xy 0;/r.6pc/:
(0,0)="(0,0)";
(1,0)="(1,0)";
(2,0)="(2,0)";
(3,0)="(3,0)";
(4,0)="(4,0)";
(5,0)="(5,0)";
(6,0)="(6,0)";
(7,0)="(7,0)";
(1,7)="(1,7)";
(2,7)="(2,7)";
(3,7)="(3,7)";
(4,7)="(4,7)";
(5,7)="(5,7)";
(6,7)="(6,7)";
(7,7)="(7,7)";
(0,1)="(0,1)";
(0,2)="(0,2)";
(0,3)="(0,3)";
(0,4)="(0,4)";
(0,5)="(0,5)";
(0,6)="(0,6)";
(0,7)="(0,7)";
(7,1)="(7,1)";
(7,2)="(7,2)";
(7,3)="(7,3)";
(7,4)="(7,4)";
(7,5)="(7,5)";
(7,6)="(7,6)";
(7,7)="(7,7)";
 "(0,0)"; "(7,0)"**\dir{.};
  "(0,1)"; "(7,1)"**\dir{.};
"(0,2)"; "(7,2)"**\dir{.};
"(0,3)"; "(7,3)"**\dir{.};
"(0,4)"; "(7,4)"**\dir{.};
"(0,5)"; "(7,5)"**\dir{.};
"(0,6)"; "(7,6)"**\dir{.};
"(0,7)"; "(7,7)"**\dir{.};
"(0,0)"; "(0,7)"**\dir{.};
"(1,0)"; "(1,7)"**\dir{.};
"(2,0)"; "(2,7)"**\dir{.};
"(3,0)"; "(3,7)"**\dir{.};
"(4,0)"; "(4,7)"**\dir{.};
"(5,0)"; "(5,7)"**\dir{.};
"(6,0)"; "(6,7)"**\dir{.};
"(7,0)"; "(7,7)"**\dir{.};
(0.5,0)="(0.5,0)";
(1.5,0)="(1.5,0)";
(2.5,0)="(2.5,0)";
(3.5,0)="(3.5,0)";
(4.5,0)="(4.5,0)";
(5.5,0)="(5.5,0)";
(6.5,0)="(6.5,0)";
(7,0.5)="(7,0.5)";
(7,1.5)="(7,1.5)";
(7,2.5)="(7,2.5)";
(7,3.5)="(7,3.5)";
(7,4.5)="(7,4.5)";
(7,5.5)="(7,5.5)";
(7,6.5)="(7,6.5)";
(2.5,6.5)*{\times}="w1";
(1.5,5.5)*{\times}="w2";
(0.5,4.5)*{\times}="w3";
(5.5,3.5)*{\circ}="w4";
(4.5,2.5)*{\circ}="w5";
(3.5,1.5)*{\circ}="w6";
(6.5,0.5)*{\circ}="w7";
"w1"; "(2.5,0)"**\dir{-};
"w1"; "(7,6.5)"**\dir{-};
"w2"; "(1.5,0)"**\dir{-};
"w2"; "(7,5.5)"**\dir{-};
"w3"; "(0.5,0)"**\dir{-};
"w3"; "(7,4.5)"**\dir{-};
"w4"; "(5.5,0)"**\dir{-};
"w4"; "(7,3.5)"**\dir{-};
"w5"; "(4.5,0)"**\dir{-};
"w5"; "(7,2.5)"**\dir{-};
"w6"; "(3.5,0)"**\dir{-};
"w6"; "(7,1.5)"**\dir{-};
"w7"; "(6.5,0)"**\dir{-};
"w7"; "(7,0.5)"**\dir{-};
(3.5,2.5)*{\bullet};
\endxy$$\,\,\,\,\,\,\,$}
\subfigure[$v_1 (j=1)$]{\label{j1}$\,\,\,\,\,\,\,$ $\xy 0;/r.6pc/:
(0,0)="(0,0)";
(1,0)="(1,0)";
(2,0)="(2,0)";
(3,0)="(3,0)";
(4,0)="(4,0)";
(5,0)="(5,0)";
(6,0)="(6,0)";
(7,0)="(7,0)";
(1,7)="(1,7)";
(2,7)="(2,7)";
(3,7)="(3,7)";
(4,7)="(4,7)";
(5,7)="(5,7)";
(6,7)="(6,7)";
(7,7)="(7,7)";
(0,1)="(0,1)";
(0,2)="(0,2)";
(0,3)="(0,3)";
(0,4)="(0,4)";
(0,5)="(0,5)";
(0,6)="(0,6)";
(0,7)="(0,7)";
(7,1)="(7,1)";
(7,2)="(7,2)";
(7,3)="(7,3)";
(7,4)="(7,4)";
(7,5)="(7,5)";
(7,6)="(7,6)";
(7,7)="(7,7)";
 "(0,0)"; "(7,0)"**\dir{.};
  "(0,1)"; "(7,1)"**\dir{.};
"(0,2)"; "(7,2)"**\dir{.};
"(0,3)"; "(7,3)"**\dir{.};
"(0,4)"; "(7,4)"**\dir{.};
"(0,5)"; "(7,5)"**\dir{.};
"(0,6)"; "(7,6)"**\dir{.};
"(0,7)"; "(7,7)"**\dir{.};
"(0,0)"; "(0,7)"**\dir{.};
"(1,0)"; "(1,7)"**\dir{.};
"(2,0)"; "(2,7)"**\dir{.};
"(3,0)"; "(3,7)"**\dir{.};
"(4,0)"; "(4,7)"**\dir{.};
"(5,0)"; "(5,7)"**\dir{.};
"(6,0)"; "(6,7)"**\dir{.};
"(7,0)"; "(7,7)"**\dir{.};
(0.5,0)="(0.5,0)";
(1.5,0)="(1.5,0)";
(2.5,0)="(2.5,0)";
(3.5,0)="(3.5,0)";
(4.5,0)="(4.5,0)";
(5.5,0)="(5.5,0)";
(6.5,0)="(6.5,0)";
(7,0.5)="(7,0.5)";
(7,1.5)="(7,1.5)";
(7,2.5)="(7,2.5)";
(7,3.5)="(7,3.5)";
(7,4.5)="(7,4.5)";
(7,5.5)="(7,5.5)";
(7,6.5)="(7,6.5)";
(3.5,6.5)*{\circ}="w1";
(1.5,5.5)*{\circ}="w2";
(0.5,4.5)*{\circ}="w3";
(5.5,3.5)*{\circ}="w4";
(2.5,2.5)*{\circ}="w5";
(4.5,1.5)*{\circ}="w6";
(6.5,0.5)*{\circ}="w7";
"w1"; "(3.5,0)"**\dir{-};
"w1"; "(7,6.5)"**\dir{-};
"w2"; "(1.5,0)"**\dir{-};
"w2"; "(7,5.5)"**\dir{-};
"w3"; "(0.5,0)"**\dir{-};
"w3"; "(7,4.5)"**\dir{-};
"w4"; "(5.5,0)"**\dir{-};
"w4"; "(7,3.5)"**\dir{-};
"w5"; "(2.5,0)"**\dir{-};
"w5"; "(7,2.5)"**\dir{-};
"w6"; "(4.5,0)"**\dir{-};
"w6"; "(7,1.5)"**\dir{-};
"w7"; "(6.5,0)"**\dir{-};
"w7"; "(7,0.5)"**\dir{-};
(2.5,6.5)*{\times};
(2.5,3.5)*{\bullet};
\endxy$ $\,\,\,\,\,\,\,$}
    \subfigure[$v_2 (j=2)$]{\label{j2}$\,\,\,\,\,\,\,$ $\xy 0;/r.6pc/:
(0,0)="(0,0)";
(1,0)="(1,0)";
(2,0)="(2,0)";
(3,0)="(3,0)";
(4,0)="(4,0)";
(5,0)="(5,0)";
(6,0)="(6,0)";
(7,0)="(7,0)";
(1,7)="(1,7)";
(2,7)="(2,7)";
(3,7)="(3,7)";
(4,7)="(4,7)";
(5,7)="(5,7)";
(6,7)="(6,7)";
(7,7)="(7,7)";
(0,1)="(0,1)";
(0,2)="(0,2)";
(0,3)="(0,3)";
(0,4)="(0,4)";
(0,5)="(0,5)";
(0,6)="(0,6)";
(0,7)="(0,7)";
(7,1)="(7,1)";
(7,2)="(7,2)";
(7,3)="(7,3)";
(7,4)="(7,4)";
(7,5)="(7,5)";
(7,6)="(7,6)";
(7,7)="(7,7)";
 "(0,0)"; "(7,0)"**\dir{.};
  "(0,1)"; "(7,1)"**\dir{.};
"(0,2)"; "(7,2)"**\dir{.};
"(0,3)"; "(7,3)"**\dir{.};
"(0,4)"; "(7,4)"**\dir{.};
"(0,5)"; "(7,5)"**\dir{.};
"(0,6)"; "(7,6)"**\dir{.};
"(0,7)"; "(7,7)"**\dir{.};
"(0,0)"; "(0,7)"**\dir{.};
"(1,0)"; "(1,7)"**\dir{.};
"(2,0)"; "(2,7)"**\dir{.};
"(3,0)"; "(3,7)"**\dir{.};
"(4,0)"; "(4,7)"**\dir{.};
"(5,0)"; "(5,7)"**\dir{.};
"(6,0)"; "(6,7)"**\dir{.};
"(7,0)"; "(7,7)"**\dir{.};
(0.5,0)="(0.5,0)";
(1.5,0)="(1.5,0)";
(2.5,0)="(2.5,0)";
(3.5,0)="(3.5,0)";
(4.5,0)="(4.5,0)";
(5.5,0)="(5.5,0)";
(6.5,0)="(6.5,0)";
(7,0.5)="(7,0.5)";
(7,1.5)="(7,1.5)";
(7,2.5)="(7,2.5)";
(7,3.5)="(7,3.5)";
(7,4.5)="(7,4.5)";
(7,5.5)="(7,5.5)";
(7,6.5)="(7,6.5)";
(2.5,6.5)*{\circ}="w1";
(3.5,5.5)*{\circ}="w2";
(0.5,4.5)*{\circ}="w3";
(5.5,3.5)*{\circ}="w4";
(1.5,2.5)*{\circ}="w5";
(4.5,1.5)*{\circ}="w6";
(6.5,0.5)*{\circ}="w7";
"w1"; "(2.5,0)"**\dir{-};
"w1"; "(7,6.5)"**\dir{-};
"w2"; "(3.5,0)"**\dir{-};
"w2"; "(7,5.5)"**\dir{-};
"w3"; "(0.5,0)"**\dir{-};
"w3"; "(7,4.5)"**\dir{-};
"w4"; "(5.5,0)"**\dir{-};
"w4"; "(7,3.5)"**\dir{-};
"w5"; "(1.5,0)"**\dir{-};
"w5"; "(7,2.5)"**\dir{-};
"w6"; "(4.5,0)"**\dir{-};
"w6"; "(7,1.5)"**\dir{-};
"w7"; "(6.5,0)"**\dir{-};
"w7"; "(7,0.5)"**\dir{-};
(1.5,5.5)*{\times};
(1.5,3.5)*{\bullet};
\endxy$ $\,\,\,\,\,\,\,$}
\subfigure[$v_3 (j=3)$]{\label{j3}$\,\,\,\,\,\,\,$ $\xy 0;/r.6pc/:
(0,0)="(0,0)";
(1,0)="(1,0)";
(2,0)="(2,0)";
(3,0)="(3,0)";
(4,0)="(4,0)";
(5,0)="(5,0)";
(6,0)="(6,0)";
(7,0)="(7,0)";
(1,7)="(1,7)";
(2,7)="(2,7)";
(3,7)="(3,7)";
(4,7)="(4,7)";
(5,7)="(5,7)";
(6,7)="(6,7)";
(7,7)="(7,7)";
(0,1)="(0,1)";
(0,2)="(0,2)";
(0,3)="(0,3)";
(0,4)="(0,4)";
(0,5)="(0,5)";
(0,6)="(0,6)";
(0,7)="(0,7)";
(7,1)="(7,1)";
(7,2)="(7,2)";
(7,3)="(7,3)";
(7,4)="(7,4)";
(7,5)="(7,5)";
(7,6)="(7,6)";
(7,7)="(7,7)";
 "(0,0)"; "(7,0)"**\dir{.};
  "(0,1)"; "(7,1)"**\dir{.};
"(0,2)"; "(7,2)"**\dir{.};
"(0,3)"; "(7,3)"**\dir{.};
"(0,4)"; "(7,4)"**\dir{.};
"(0,5)"; "(7,5)"**\dir{.};
"(0,6)"; "(7,6)"**\dir{.};
"(0,7)"; "(7,7)"**\dir{.};
"(0,0)"; "(0,7)"**\dir{.};
"(1,0)"; "(1,7)"**\dir{.};
"(2,0)"; "(2,7)"**\dir{.};
"(3,0)"; "(3,7)"**\dir{.};
"(4,0)"; "(4,7)"**\dir{.};
"(5,0)"; "(5,7)"**\dir{.};
"(6,0)"; "(6,7)"**\dir{.};
"(7,0)"; "(7,7)"**\dir{.};
(0.5,0)="(0.5,0)";
(1.5,0)="(1.5,0)";
(2.5,0)="(2.5,0)";
(3.5,0)="(3.5,0)";
(4.5,0)="(4.5,0)";
(5.5,0)="(5.5,0)";
(6.5,0)="(6.5,0)";
(7,0.5)="(7,0.5)";
(7,1.5)="(7,1.5)";
(7,2.5)="(7,2.5)";
(7,3.5)="(7,3.5)";
(7,4.5)="(7,4.5)";
(7,5.5)="(7,5.5)";
(7,6.5)="(7,6.5)";
(2.5,6.5)*{\circ}="w1";
(1.5,5.5)*{\circ}="w2";
(3.5,4.5)*{\circ}="w3";
(5.5,3.5)*{\circ}="w4";
(0.5,2.5)*{\circ}="w5";
(4.5,1.5)*{\circ}="w6";
(6.5,0.5)*{\circ}="w7";
"w1"; "(2.5,0)"**\dir{-};
"w1"; "(7,6.5)"**\dir{-};
"w2"; "(1.5,0)"**\dir{-};
"w2"; "(7,5.5)"**\dir{-};
"w3"; "(3.5,0)"**\dir{-};
"w3"; "(7,4.5)"**\dir{-};
"w4"; "(5.5,0)"**\dir{-};
"w4"; "(7,3.5)"**\dir{-};
"w5"; "(0.5,0)"**\dir{-};
"w5"; "(7,2.5)"**\dir{-};
"w6"; "(4.5,0)"**\dir{-};
"w6"; "(7,1.5)"**\dir{-};
"w7"; "(6.5,0)"**\dir{-};
"w7"; "(7,0.5)"**\dir{-};
(0.5,4.5)*{\times};
(0.5,3.5)*{\bullet};
\endxy$$\,\,\,\,\,\,\,$}
  \end{center}
  \caption{MT-moves using different $j's$}
  \label{branching}
\end{figure}

\begin{corollary}\label{algo}Let $w,u$ be two permutations both with $\ell(c(w))=\ell(c(u))=m$ (for the case when $\ell(c(w))\neq\ell(c(u))$, add enough ones to the front of one permutation). Assume $u$ is Grassmannian. Apply  MT-moves successively to $D(1^{m}\times w\times u)$. Stop applying MT-moves to a diagram $D$ as soon as all the boxes in its diagram are in the first $2m$ rows. Denote the multiset of the diagrams obtained this way by $A$. Then in the canonical expansion (\ref{expansion}) $F_w\cdot F_u=\sum_{v\in V}c_{wu}^vF_v$, we have $$c_{wu}^v=\#\{D\in A\mid D=D(v)\}.$$
\end{corollary}

\section{Other Stable Expansions}
In this section, we study some other Stable expansions related with Schubert polynomials. Given a unique expansion, we study the behavior of that expansion when we embed $w\mapsto 1^n\times w$, as we did for Theorem \ref{main} and Theorem \ref{full}. We call the eventually stabilized behavior as described in Theorem \ref{main} \emph{weak stable property}; and if it further satisfies the property that once there are no new terms, there will be no new terms ever, as described in Theorem \ref{full}, we call it \emph{strong stable property}. First, as a direct corollary of Theorem \ref{main}, we have
\begin{corollary}\label{more} For the unique expansion of the product of finitely many Schubert polynomials into Schubert polynomials, we have the weak stable property, i.e., \begin{enumerate}
\item  Suppose $\s_{w_1}\cdots \s_{w_\ell}=\sum_{v_0\in V_0}c_{w_1,\dots,w_\ell}^{v_0}\s_{v_0}$.
Then $$\s_{1\times w_1}\cdots \s_{1\times w_\ell}=\sum_{v_0\in V_0}c_{w_1,\dots,w_\ell}^{v_0}\s_{1\times
v_0}+\sum_{v_1\in V_1}c_{w_1,\dots,w_\ell}^{v_1}\s_{v_1},$$ where $v_1(1)\neq 1$, for each $v_1\in V_1$.
\item Let $k=\ell(w_1)+\cdots+\ell(w_\ell)$. Then for all $n\ge k$, we have
$$\mathfrak{S}_{1^{n}\times w_1}\cdots \mathfrak{S}_{1^{n}\times w_\ell}=\sum_{v_0\in V_0}c^{v_0}_{w_1,\dots,w_\ell}\s_{1^{n}\times v_0}+\sum_{v_1\in V_1}c^{v_1}_{w_1,\dots,w_\ell}\s_{1^{n-1}\times v_1}+\cdots+\sum_{v_k\in V_k}c^{v_k}_{w_1,\dots,w_\ell}\s_{1^{n-k}\times v_k},$$
where $V_i$ (possibly empty) is the set of new permutations appearing in $\s_{1^i\times w_1}\cdots\s_{1^i\times w_\ell}$ compared to $\s_{1^{i-1}\times w}\cdot\s_{1^{i-1}\times u}$.
\end{enumerate}
\end{corollary}
\subsection{Product of Double Schubert polynomials}
For the Double Schubert polynomials, we have the following connection to Schubert polynomials (for example, see Prop 2.4.7 in \cite{schbook})
$$\s_w(x,y)=\sum_{\substack{{w=v^{-1}u}\\{\ell(w)=\ell(u)+\ell(v)}}}\s_u(x)\s_v(-y).$$
Then consider the product of two double Schubert polynomials
$$\s_{1^n\times w}(x,y)\s_{1^n\times w'}(x,y)=
\sum_{\substack{{w=v^{-1}u}\\{\ell(w)=\ell(u)+\ell(v)}\\{w'=v'^{-1}u'}\\{\ell(w')=\ell(u')+\ell(v')}}}\s_{1^n\times u}(x)\s_{1^n\times v}(-y)\s_{1^n\times u'}(x)\s_{1^n\times v'}(-y).$$
By Corollary \ref{more}, we have
\begin{corollary}For the unique expansion of the product of finitely many double Schubert polynomials into Schubert polynomials, we have the weak stable property.
\end{corollary}

\subsection{Stable expansion between $\s_w$ and $e_I$}
We write (\ref{combdef}) as $$\s_{w}=\sum_{a\in \N^{\infty}}K_{w,a}X^a,$$
where each $K_{w,a}$ is nonnegative integers, and $K=(K_{w,a})$ is known as the Schubert-Kostka matrix.
Let $$e_i^k=\sum_{1\le r_1<\cdots<r_i\le k}x_{r_1}\cdots x_{r_i},$$
and for $I=(i_1,i_2,\dots,i_n)$, let $$e_I=e_{i_1}^1e_{i_2}^2\cdots e_{i_n}^n.$$
Notice that $e_i^k=0$, if $i>k$, so we require $i_k\le k$ in $I$.

The following result is well known.
\begin{proposition}[{see \cite{poly}, \cite[(2.6)-(2.7)]{fonct}, \cite[(4.13)]{non}}]\label{basis}We have the following $\Z$-linear bases for $\Z[x_1,\dots,x_n]/ I_n$, and each of them spans the same vector space which is complementary to $I_n$.
\begin{enumerate}
\item the monomials $x_1^{a_1}\cdots x_{n-1}^{a_{n-1}}$ such that $0\le a_k\le n-k$;
\item the standard elementary monomials $e_{i_1i_2\dots i_{n-1}}$;
\item the Schubert polynomials $\s_w$ for $w\in S_n$.
\end{enumerate}
\end{proposition}
By Proposition \ref{basis}, we have the unique expansions of $e_I$ into $\s_w$ and $\s_w$ into $e_I$. In this subsection, we will prove the stable property for these two unique expansions.
\begin{proposition}\label{etos}For $I=(i_1,\dots,i_n)$, we have the strong stable property for the expansion $$e_I=\sum_{w\in W}\beta_w^I\s_w,$$ i.e., for all $k\ge r=i_1+i_2+\cdots+i_n-n$, we have
 $$e_{(0^k,i_1,i_2,\dots,i_n)}=\sum_{w\in W}\beta_w^I\s_{1^k\times w}+\sum_{w_1\in W_1}\beta_{w_1}^{I}\s_{1^{k-1}\times w}+\cdots++\sum_{w_k\in W_r}\beta_{w_r}^I\s_{1^{k-r}\times w},$$
 where for $i=1,\dots,r$, we have $W_i\neq\emptyset$ and $W_r$ is the single permutation $w_r=23\cdots(r+n+1)1$. Moreover, for $1\le k<r$, $W_k$ is the set of new permutations added in the expansion of $e_{(0^k,i_1,i_2,\dots,i_n)}$ from $e_{(0^{k-1},i_1,i_2,\dots,i_n)}$.
\end{proposition}
To prove Proposition \ref{etos}, we need a lemma, which uses the  Pieri rule.
\begin{proposition}[Pieri rule]\label{pieri}Define the following operator on Schubert polynomials:
\begin{equation*}
T_{i,j}\s_w=\begin{cases} \s_{wt_{i,j}} & \text{ if } \ell(wt_{i,j})=\ell(w)+1,\\
0 & \text{ otherwise}.
\end{cases}
\end{equation*}
We have
$$e_r^k\s_w=\sum T_{i_1,j_1}T_{i_2,j_2}\dots T_{i_r,j_r}\s_w,$$
summing over $i_1,\dots,i_r\le k<j_1,\dots,j_r$, such that
$i_1,\dots,i_r$ are all distinct and $j_1\le \dots\le j_r$.
\end{proposition}

\begin{lemma}\label{onee}The unique expansion of $e_i^j\s_w=\sum_{w\in W}\beta_w^I\s_{w}$ into Schubert polynomials has strong stable property, i.e.,
there exist $r$ such that for all $k\ge r$, we have
 $$e_i^{j+k}\s_{1^k\times w}=\sum_{w\in W}\beta_w^I\s_{1^k\times w}+\sum_{w_1\in W_1}\beta_{w_1}^{I}\s_{1^{k-1}\times w}+\cdots++\sum_{w_k\in W_r}\beta_{w_r}^I\s_{1^{k-r}\times w},$$
 where for $i=1,\dots,r$, we have $W_i\neq\emptyset$. Moreover, for $1\le k<r$, $W_k$ is the set of new permutations added to the expansion
 of $e_i^{j+k}\s_{1^k\times w}$ compared to the expansion of $e_i^{j+k-1}\s_{1^{k-1}\times w}$.
\end{lemma}
\begin{proof}Use Proposition \ref{pieri}. Let $m$ be the minimal number of simple transformations we need in order to move the letter $1$ to a position after $j$ in $w$. In the next step, we
consider $e_{i}^{j+1}\s_{1\times w}$. We will get new terms if we can exchange the letter $1$ in $1\times w$ with some other letter. Then we can show that
 $e_{i}^{j+k}\s_{1^k\times w}$ will have new terms in the expansion for all $0<k\le i-m$ and there will be
 no new terms if $k>i-m$. In step $k=i-m$, there is exactly one new permutation.
\end{proof}

\begin{proof}[proof of Proposition \ref{etos}]Use Lemma \ref{onee} from left to right, with $e_{i_1}^1=\s_w$, $$((e_{i_1}^1e_{i_2}^2)e_{i_3}^3)\cdots e_{i_n}^n,$$ 
 For $e_{i_1}^1e_{i_2}^2$, let $\s_{w_1}$ be the last new term added at step $n_1=i_2-1$. Then consider
$\s_{w_1}e_{i_3}^{k+n_1}$, let $\s_{w_2}$ be the last new term at step $n_1+n_2$, with $n_2=i_3-1$, etc. We have the last new term added at step
$N=n_1+\dots+n_n-n$, is $\s_{23\dots (|I|+1)1}$. Before this step, there are always new terms being added to the expansion.
\end{proof}

For the expansion of $\s_w$ into $e_I$, we have the following stable property:

\begin{proposition}\label{stoe}For the unique expansion $\s_w=\sum_{I\in N^{\infty}}a_I^we_I$, we have weak stable property, i.e.,
in the expansion $\s_{1\times w}=\sum_{J\in N^{\infty}}b_J^{1\times w}e_J$, we have
$$b_{0I}^{1\times w}=a_{I}^{w}.$$ 
$\s_{1\times w}=\sum_{J\in N^{\infty}}b_J^{1\times w}e_J$.
\end{proposition}
To prove this stable property, we use the following two lemmas:
\begin{lemma}[{\cite[page 31]{ps}}]\label{ps1}For $w=w_1\cdots w_n\in S_n$, we have
$$\s_{ww_0}=\sum_{a}K_{a,w}^{-1}e_{w_0(\rho_n-a)},$$ where $w_0=n(n-1)\cdots 1$,
$\rho_n=(n-1,n-2,\cdots,1,0)$, and $K^{-1}=(K_{a,w}^{-1})$ is the inverse of $K$.
\end{lemma}
\begin{lemma}[{\cite[Proposition 17.3]{ps}}]\label{ps2}For any $u\in S_n$ and $a\in \N^{\infty}$, we have
$$K_{a,u}^{-1}=\sum_{w\in S_n}(-1)^{\ell(w)}K_{w_0u,w(\rho_n)-a},$$
where $w$ acts on a vector as rearranging the coordinates, e.g., $312(2,1,0)=(0,2,1)$.
\end{lemma}

\begin{proof}[Proof of Proposition \ref{stoe}]Consider Lemma \ref{ps1}.
Notice that $w_0(\rho_n-a)=(0-a_n,1-a_{n-1},\dots,n-1-a_1)$, and $w_0(\rho_{n+1}-b)=(0-b_{n+1},1-b_n,\dots,n-b_1)$.
Assume $w_0(\rho_{n+1}-b)=(0,w_0(\rho_n-a))$. Then $b=(a+1,0)$. Assume $1\times ww_0=uw_0$. Then $u=(n+1)\times w$.
So to prove the result, it suffices to show that
$$K^{-1}_{(a+1,0),(1+n)\times u}=K^{-1}_{a,u}.$$
Then by Lemma \ref{ps2}, we have
$$K_{(a+1,0),(1+n)\times u}^{-1}=\sum_{w\in S_{n+1}}(-1)^{\ell(w)}K_{w_0((1+n)\times u),w(\rho_{n+1})-(a+1,0)}.$$
Notice that $w_0((1+n)\times u)=u$ and $w(\rho_{n+1})-(a+1,0)=w(\rho_n)-a$ if $w(\rho_{n+1})_{n+1}=0$.
But $w(\rho_{n+1})_{n+1}\neq 0$, $w(\rho_{n+1})-(a+1,0)\notin \N^{\infty}$, so $K_{w_0((1+n)\times u),w(\rho_{n+1})-(a+1,0)}=0$.
Therefore, $K^{-1}_{(a+1,0),(1+n)\times u}=K^{-1}_{a,u}$.\qedhere
\end{proof}
\begin{remark}\label{2p}
Consider the expansion of two Schubert polynomials $\s_w\s_u$ into Schubert polynomials again as we studied in Theorem \ref{main}. By Proposition \ref{stoe}, we can get a stabilized expansion of $\s_w$ into the $e_I$'s.  Then, by Lemma \ref{onee}, the expansion of each term $e_I\s_u$ into Schubert polynomials stabilizes. This way, we get a second proof of Theorem \ref{main}.
\end{remark}
\subsection*{Acknowledgements.}
I thank Richard Stanley, Alex Postnikov and Steven Sam for
helpful discussions. I am also very grateful for the anonymous referee from Fpsac 2012 for some very good suggestion.

\bigskip

\filbreak \noindent Nan Li\\
Department of Mathematics\\
Massachusetts Institute of Technology\\
Cambridge, MA 02139\\
{\tt nan@math.mit.edu}

\end{document}